\documentclass{amsart}
\usepackage[margin=2.5cm]{geometry}
\usepackage[svgnames]{xcolor}
\usepackage{enumitem}
\usepackage{amsmath,amssymb}
\usepackage{framed}
\usepackage{tikz}
\usepackage{dsfont}
\usepackage[latin1]{inputenc}
\usepackage[english]{babel}
\usepackage{multimedia}
\usepackage{multicol}
\usepackage[T1]{fontenc}
\usepackage{times}
\usepackage{graphicx}
\usepackage{tabmac}

\newcommand{\tcercle}[1]{\ensuremath{\setlength{\unitlength}{1ex}\begin{picture}(2.8,2.8)\put(1.4,1.4){\circle{2.8}\makebox(-5.6,0){#1}}\end{picture}}}

\newcommand{\ct}[2]{\textrm{ct}({#1,#2})}

\newcommand{\omitt}[1]{}

\newcommand{\Z}{\mathbb Z}
\newcommand{\Q}{\mathbb Q}

\DeclareMathOperator{\comp}{comp}
\DeclareMathOperator{\sign}{sign}
\DeclareMathOperator{\Rev}{rev}
\DeclareMathOperator{\size}{n}
\newcommand{\ga}{\alpha}
\newcommand{\gb}{\beta}
\newcommand{\fund}[1]{L_{#1}}

\newcommand{\funda}{\fund{\ga}}
\newcommand{\te}{\theta}

\def\newShuf{fundamental }
\def\pl{[}
\def\pr{]}

\usepackage[colorinlistoftodos]{todonotes}
\newcommand{\Luc}[1]{\todo[size=\tiny,inline,color=yellow!30]{#1
      \\ \hfill --- Luc}}
\newcommand{\Susanna}[1]{\todo[size=\tiny,inline,color=red!30]{#1
    \\ \hfill --- Susanna}}

\newcommand{\lltodo}[1]{\Luc{#1}}

\newcommand{\lucNEW}[1]{\textcolor{blue}{{\it #1}}}

\def\defstyle{\bf }

\numberwithin{equation}{section}

\theoremstyle{definition}
  \newtheorem{theorem}{Theorem}[section]
  
  \newtheorem{proposition}[theorem]{Proposition}
  
  \newtheorem{claim}[theorem]{Claim}
\newtheorem*{proposition*}{Proposition}
  \newtheorem {corollary}[theorem]{Corollary}
  \newtheorem{definition}[theorem]{Definition}
  \newtheorem{example}[theorem]{Example}
\newcommand{\compgeq}{\succcurlyeq}
\theoremstyle{remark}
\newtheorem{remark}[theorem]{Remark}

\newsavebox{\smlmat}
\savebox{\smlmat}{$\left(\begin{smallmatrix}1&2&3&4&5&6\\2&4&5&3&6&1\end{smallmatrix}\right)$}
\newsavebox{\smlmatInv}
\savebox{\smlmatInv}{$\left(\begin{smallmatrix}1&2&3&4&5&6\\6&1&4&2&3&5\end{smallmatrix}\right)$}

\title{Fundamental quasisymmetric functions in superspace}

\author{Susanna Fishel}
\address{School of Mathematical and Statistical Sciences, Arizona State University, P.O. Box 871804, Tempe, AZ 85287-1804, USA} \email{sfishel1@asu.edu}

\author{Jessica Gatica}
\address{Instituto de Matem\'atica y F\'{\i}sica, Universidad de Talca, Casilla 747, Talca,
Chile} \email{jgatica@utalca.cl}

\author{Luc Lapointe}
\address{Instituto de Matem\'atica y F\'{\i}sica, Universidad de Talca, Casilla 747, Talca,
  Chile} \email{llapointe@utalca.cl}

\author{Mar\'{\i}a Elena Pinto}
\address{Instituto de Matem\'atica y F\'{\i}sica, Universidad de Talca, Casilla 747, Talca,
Chile} \email{mepinto@utalca.cl}

\date{\today}
\begin{document}
\thanks{Funding: this work was supported by
the Fondo Nacional de Desarrollo Cient\'{\i}fico y Tecnol\'ogico de Chile (FONDECYT) Initiation into Research  Grant \#11200527 (J.G.)
and Regular Grant \#1210688 (L.L.), and by Simons Collaboration Grants for
Mathematicians \#359602 and 709671 (S.F.)}
\maketitle
\begin{abstract} The fundamental quasisymmetric functions in superspace
are a generalization of the fundamental quasisymmetric functions  involving anticommuting variables.
We obtain the action of the product,
coproduct, and antipode on the  fundamental quasisymmetric functions in superspace. We also extend to superspace the well known expansion of the Schur functions in terms of fundamental quasisymmetric functions.
\end{abstract}

\section{Introduction}
 Symmetric functions in superspace originated from the study of the
 supersymmetric generalization of the trigonometric
 Calogero-Moser-Sutherland model \cite{DLM1,DLM2,DLM3,BDLM}.  In the superspace setting, the polynomials $f(x,
 \theta)$, where $(x,\theta) =(x_1, \ldots, x_N, \theta_1,\dots
 \theta_N)$, not only depend on the usual commuting variables
 $x_1,\ldots, x_N$ but also on the anticommuting variables
 $\theta_1,\ldots, \theta_N$, where
 $\theta_i\theta_j=-\theta_j\theta_i$, and $\theta^2_i=0$.
 In \cite{FLP}, the first, third and fourth author of the current paper
 extended
 the theory of symmetric functions in superspace by introducing 
 quasisymmetric and noncommutative
 symmetric functions in superspace, especially focusing on the corresponding Hopf algebra structures. In particular, the action of the product,
 coproduct, and antipode on the monomial quasisymmetric functions in superspace basis was obtained. In this paper, our goal is to study in depth the 
 fundamental quasisymmetric functions in superspace, $L_\alpha$, which were only defined in passing in \cite{FLP}. It turns out that the action of the product,
 coproduct, and antipode on the  fundamental quasisymmetric functions in superspace  is also quite elegant.

With only commuting
variables, the expansion of fundamental quasisymmetric functions in
monomial quasisymmetric functions is described using compositions --- in
superspace we need dotted compositions to account for the anticommuting
variables. In \cite{FLP}, we introduced two different partial orders on
dotted compositions and used them to define
two families of fundamental quasisymmetric
functions in terms of their expansion by monomial quasisymmetric
functions in superspace: the  fundamental and cofundamental functions. 
Note that it is not totally surprising that there are two families of fundamental quasisymmetric functions in superspace as there are also two families of Schur functions in superspace \cite{JL}.  Indeed, extending to superspace the natural expansion of the Schur functions in terms of fundamental quasisymmetric functions (as will be decribed at the end of the introduction in the case of the fundamentals)  was a key element in defining
the fundamental and cofundamental quasisymmetric functions in superspace. 
In this article, we will however only study the fundamental
quasisymmetric functions in superspace for the simple reason that the cofundamentals  behave badly (we have not been able for instance to obtain explicit expressions for the action of the  product and antipode on the cofundamentals).

We begin this article with some review of material from \cite{FLP}, then
define a correspondence between dotted compositions and sets. It is
reminiscent of the correpondence between a composition and its set of
partial sums. We are able to simplify and unify proofs using these
sets and rewrite both monomial and fundamental quasisymmetric
functions in superspace in terms of them.

In Section~\ref{sec:coprodOfFunds}, we show that the coproduct acts quite nicely
on the fundamental basis and the proof of
Proposition~\ref{claimcoprod} is straightforward. The product also
acts well (Section~\ref{sec:prodOfFunds}) as it expands as a
signed sum of fundamentals over a set of fundamental shuffles of dotted
compositions, similar to the case with only commuting
variables. However, fundamental shuffles are more complex than the
overlapping shuffles needed in \cite{FLP}, which are in turn more
complex than the shuffles in \cite{GR} and \cite{LMW}. The proof of Proposition~\ref{prop:fundProd} is correspondingly intricate.

In Section~\ref{sec:antipodes}, we turn to the antipode.  The action of the antipode on $L_\alpha$ happens to be structurally much more involved than in the case with only commuting variables.  Fortunately, we are able to 
provide an efficient algorithm to compute the antipode by 
introducing two operations on the space of quasisymmetric functions in superspace (see \eqref{eqdecomp} and Proposition~\ref{propcolumn}).   The deeper result of this section is the compatibility theorem between the antipode and our operations (Theorem~\ref{propantiM}).

Finally, in Section~\ref{sec:schurFnsInTermsOfFunds} we show that the
beautiful connection between Schur functions and fundamental quasisymmetric functions extends naturally to superspace (this can be seen as a confirmation
that the definition of fundamental quasisymmetric functions in superspace is indeed the right one).  To be more precise,   Proposition~\ref{propsymfun} states that
\begin{equation}
s_{\Lambda/\Omega}= \sum_{\mathcal T}  (-1)^{{\rm inv}(\mathcal T)} L_{\comp(\mathcal T)} ,
\end{equation}
where the sum is over a certain set of tableaux of shape 
${\Lambda/\Omega}$, with $\Lambda$ and $\Omega$ 
superpartitions (dotted partitions).

\section{Preliminaries}
\subsection{Monomial quasisymmetric functions in superspace}

\begin{definition}
  A {\defstyle dotted composition} $(\alpha_1,\alpha_2,\dots,\alpha_l)$
  is a vector whose entries either belong to  $\{1,2,3,\dots\}$ or 
  to $\{\dot 0, \dot 1,\dot 2,\dots\}$.
The {\defstyle length} of $\ga$,
denoted $\ell(\ga)$, is the number of parts $l$ of $\alpha$.
We define the
sequence $\eta=\eta(\ga)=(\eta_1,\ldots,\eta_{\ell(\ga)})$ by
\begin{equation}\label{D:eta}
  \eta_i=\begin{cases}1&\text{if $\ga_i$ is dotted,}\\0&\text{otherwise.}\end{cases}
\end{equation}
We let $|\alpha|:=\alpha_1+\cdots +\alpha_l$
be the {\defstyle total degree} of $\alpha$ (in the sum, the dotted entries
are considered as if they did not have dots on them).
The number of dotted parts of $\alpha$ is called the {\defstyle fermionic degree} of $\ga$. We write $\alpha\vdash(n,m)$ if $\alpha$ has total degree $n$ and fermionic degree $m$ and we write $x^{\ga_i}_j$ whether
 $\ga_i$ is dotted or not.
\end{definition}

\omitt{\lucNEW{Why cite}   \lucNEW{and} \cite{LMW} ?}
The fundamental quasisymmetric functions were first defined in \cite{Ge}. See 
\cite{LMW} and \cite{GR} for more recent expositions.

\begin{definition}
  \label{D:qs}
Let $\mathcal R(x,\theta)$ be the ring of formal power series of finite degree in
$\Q[[x_1,x_2,\ldots,\te_1,\te_2,\ldots]]$.
The {\defstyle quasisymmetric functions in superspace} ${\rm sQSym}$ 
will be the $\mathbb Q$-vector space of the elements
$f$ of $\mathcal R(x,\theta)$ such that
for every dotted compositions
  $\ga=(\ga_1,\ldots,\ga_{l})$ with $\eta=\eta(\ga)$ as in
  \eqref{D:eta}, all monomials
  $\te^{\eta_1}_{i_1}\cdots\te^{\eta_{l}}_{i_l} x^{\ga_1}_{i_1}\cdots
  x^{\ga_{l}}_{i_{l}}$ in $f$ with indices $i_1<\cdots
  <i_{l}$ have the same coefficient.
\end{definition}

There is a natural basis of  ${\rm sQSym}$ provided by the
generalization of the monomial quasisymmetric functions to superspace.
\begin{definition}{\cite[Definition 2.3]{FLP}}
  \label{D:msq}
 Let $\alpha$ be a dotted composition with $\ell(\ga)=l$.  Then the {\defstyle monomial quasisymmetric
   function in superspace} $M_\ga$ is defined
 as $$M_{\alpha}=\sum_{i_1<i_2<\cdots<i_l}
 \te_{i_1}^{\eta_1} \te_{i_2}^{\eta_2} \cdots \te_{i_l}^{\eta_l} x_{i_1}^{\ga_1}\cdots
 x_{i_l}^{\ga_l},$$
 where $\eta=\eta(\ga)$.

 \omitt{$\eta_i = \left \{ \begin{array}{c@{\qquad}l}
                  1 & \text{if} \quad i\in \{j_1,j_2,\ldots, j_k\}\text{i.e., if $\ga_i$ is dotted, and} \\
                  0 & \text{otherwise}
                \end{array} \right.$}
\end{definition}
\begin{example}
 Restricting to four variables, we have
 $$M_{\dot{3},1,2}(x_1,x_2,x_3,x_4;\te_1,\te_2,\te_3,\te_4)=\te_1x_1^3x_2x_3^2+\te_1x_1^3x_2x_4^2+\te_1x_1^3x_3x_4^2+\te_2x_2^3x_3x_4^2,$$
and
$$M_{3,\dot{1},\dot{2}}(x_1,x_2,x_3,x_4;\te_1,\te_2,\te_3,\te_4)=\te_2\te_3 x_1^3x_2x_3^2+\te_2\te_4 x_1^3x_2x_4^2+\te_3\te_4 x_1^3x_3x_4^2+\te_3\te_4 x_2^3x_3x_4^2.$$
\end{example}

It was shown in \cite{AGM} that ${\rm sQSym}$ is a Hopf algebra with coproduct
$\Delta : {\rm sQSym}\to {\rm sQSym} \otimes {\rm sQSym}$ corresponding to the duplication of alphabets:
\begin{equation} \label{dupli}
(\Delta f)(x,y; \theta, \phi) = f(x,y; \theta, \phi) ,
 \end{equation}  
where $f(x,y; \theta, \phi)$ is considered an element of ${\rm sQSym} \otimes {\rm sQSym}$ and where the variables in the alphabets $\theta$ and $\phi$
anticommute.  The action of the coproduct on monomial quasisymmetric functions in superspace turns out to be quite simple.
\begin{proposition}{\cite[Proposition 5.7]{FLP}}\label{coprod} Let $\ga=(\ga_1,\ldots , \ga_l)$ be a dotted compositions, then
$$\Delta(M_{\ga})=  \sum_{k=0}^{l} M_{(\ga_1, \ldots , \ga_k)}\otimes  M_{(\ga_{k+1}, \ldots , \ga_l)}. $$

\end{proposition}

\begin{example}
$$\Delta(M_{(\dot{2},1,\dot{3},4)})= 1 \otimes M_{(\dot{2},1,\dot{3},4)}+ M_{(\dot{2})}\otimes M_{(1,\dot{3},4)} +  M_{(\dot{2},1)}\otimes M_{(\dot{3},4)} +  M_{(\dot{2},1,\dot{3})}\otimes M_{(4)} +  M_{(\dot{2},1,\dot{3},4)}\otimes 1 .$$
\end{example}

\subsection{Partial orders on compositions}

We will need two partial orders on dotted compositions. Given
compositions $\ga$ and $\gb$, we say that $\ga$ covers $\gb$ in the
first partial order, written $\gb \preccurlyeq \ga$, if we can obtain $\ga$
by adding together a pair of adjacent non-dotted parts of $\gb$. The
first partial order is the transitive closure
of this cover relation.  If  $\gb \preccurlyeq \ga$ we say that
{\defstyle $\gb$ strongly refines $\ga$} or that $\ga$ {\defstyle strongly coarsens $\gb$}.

The second partial order on dotted compositions, is generated by the
following covering relation: 
$\ga$ covers $\gb$, written 
$\gb\trianglelefteq \ga$, if we can
obtain $\ga$ by adding together two adjacent parts of $\gb$,  where at most one of these two parts is  dotted. {If one of the parts in $\gb$ is dotted, then their sum in $\ga$ must be dotted.} If $\gb\trianglelefteq \ga$ we say this time that
{\defstyle $\gb$ weakly refines $\ga$} or that $\ga$ {\defstyle weakly coarsens $\gb$}.

Please see Figure~\ref{Fig:twoPosets} for the two orders.
When no parts are dotted, both covering relations become the covering
 relation on compositions described in \cite{LMW} and \cite{GR}.  Also note that it is immediate that  $\gb \preccurlyeq \ga$ implies  $\gb\trianglelefteq \ga$.

\begin{example} We have 
  $(1,1,\dot{3},2,1,1) \preccurlyeq 
(2,\dot{3},2,2) \preccurlyeq   (2,\dot{3},4) $ while
$ (1,1,\dot{2},2,\dot{3})  \trianglelefteq  (\dot{4},2,\dot{3})$ and
$ (1,1,\dot{2},2,\dot{3})  \trianglelefteq  (2,\dot{2},\dot{5})  \trianglelefteq  (\dot{4},\dot{5})$.
\end{example}

\begin{center}
\begin{figure}[ht]
\begin{tikzpicture}[font=\small]
  \def\a{2}
  \def\b{1.5}
  \node (11212) at (0,1*\b) {$(1,1,\dot{2},1,2)$};
  \node (2212) at (-1,2*\b) {$(2,\dot{2},1,2)$};
  \node (1123) at (1,2*\b) {$(1,1,\dot{2},3)$};
  \node (223) at (0,3*\b) {$(2,\dot{2},3)$};

  \draw[thick, blue] (11212)--(2212);
  \draw[thick, blue] (11212)--(1123);

   \draw[thick, blue] (223)--(2212);
   \draw[thick, blue] (223)--(1123);

\begin{scope}[shift={(4.5*\a,0)}]

\node (11212) at (0,0*\b) {$(1,1,\dot{2},1,2)$};

\node (2212) at (-1.5*\a,1*\b) {$(2,\dot{2},1,2)$};
\node (1312) at (-.5*\a,1*\b) {$(1,\dot{3},1,2)$};
\node (1132) at (.5*\a,1*\b) {$(1,1,\dot{3},2)$};
\node (1123) at (1.5*\a,1*\b) {$(1,1,\dot{2},3)$};
 
\node (412) at (-2.5*\a,2*\b) {$(\dot{4},1,2)$};
\node (232) at (-1.5*\a,2*\b) {$(2,\dot{3},2)$};
\node (223) at (-.5*\a,2*\b) {$(2,\dot{2},3)$};
\node (115) at (.5*\a,2*\b) {$(1,1,\dot{5})$};
\node (142) at (1.5*\a,2*\b) {$(1,\dot{4},2)$};
\node (133) at (2.5*\a,2*\b) {$(1,\dot{3},3)$};

\node (52) at (-1.5*\a,3*\b) {$(\dot{5},2)$};
\node (43) at (-.5*\a,3*\b) {$(\dot{4},3)$};
\node (25) at (.5*\a,3*\b) {$(2,\dot{5})$};
\node (16) at (1.5*\a,3*\b) {$(1,\dot{6})$};

\node (7) at (0*\a,4*\b) {$(\dot{7})$};

\draw[thick, blue] (11212)--(2212);
\draw[thick, red] (11212)--(1312);
\draw[thick, red] (11212)--(1132);
\draw[thick, blue] (11212)--(1123);

\draw[thick, red] (412)--(2212);
\draw[thick, red] (232)--(2212);
\draw[thick, blue] (223)--(2212);
\draw[thick, red] (412)--(1312);
\draw[thick, red] (142)--(1312);
\draw[thick, blue] (133)--(1312);
\draw[thick, blue] (232)--(1132);
\draw[thick, red] (115)--(1132);
\draw[thick, red] (142)--(1132);
\draw[thick, blue] (223)--(1123);
\draw[thick, red] (133)--(1123);
\draw[thick, red] (115)--(1123);

\draw[thick, red] (412)--(52);
\draw[thick, red] (232)--(52);
\draw[thick, red] (223)--(43);
\draw[thick, blue] (412)--(43);
\draw[thick, red] (142)--(52);
\draw[thick, red] (133)--(43);
\draw[thick, red] (232)--(25);
\draw[thick, blue] (115)--(25);
\draw[thick, red] (115)--(16);
\draw[thick, red] (142)--(16);
\draw[thick, red] (223)--(25);
\draw[thick, red] (133)--(16);

\draw[thick, red] (7)--(52);
\draw[thick, red] (7)--(43);
\draw[thick, red] (7)--(25);
\draw[thick, red] (7)--(16);

\end{scope}

  \end{tikzpicture}
  \caption{The poset on the left is all dotted compositions above
  $(1,1,\dot{2},1,2)$ using the first partial partial order ($\preccurlyeq$). On the
  right the poset is again all dotted compositions above $(1,1,\dot{2},1,2)$,
  but using the second partial partial order ($\trianglelefteq$). Note that on the right,  the covers that are also covers in the partial order ($\preccurlyeq$) are in blue. }
\label{Fig:twoPosets}
\end{figure}
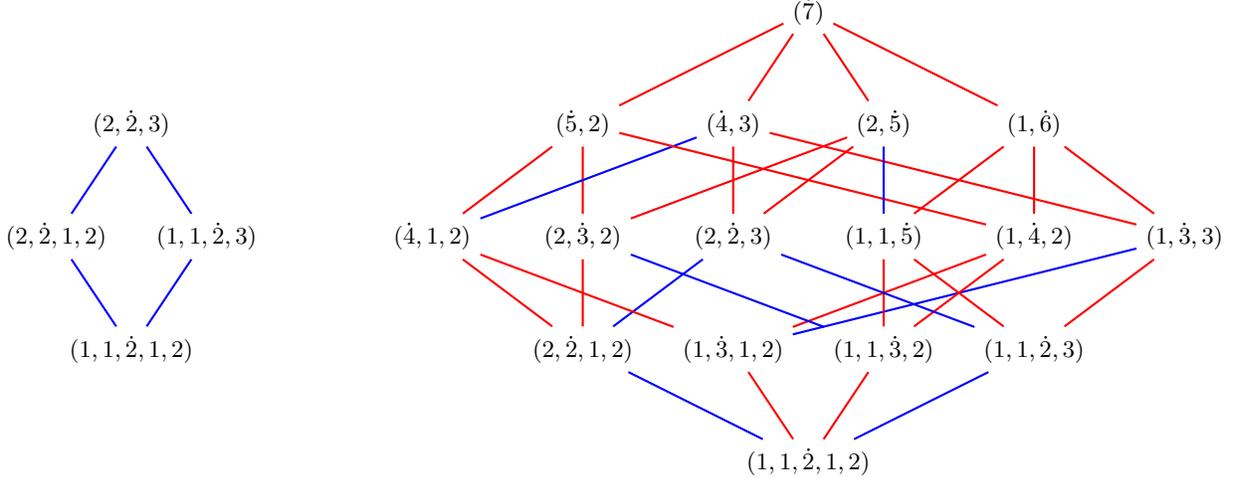
\end{center}

\subsection{Antipode}
{ We now review the action of the antipode $S: {\rm sQSym} \to {\rm
   sQSym}$, defined explicitly on quasisymmetric monomial functions in
 superspace in \cite{FLP}.}

Let the {\defstyle reverse} of a composition $\ga=(\ga_1,\ldots,\ga_l)$ be
$\Rev(\ga)=(\ga_l,\ldots,\ga_1)$.
 \begin{proposition}{\cite[Proposition 5.10]{FLP}}
\label{prop:antiM}
  Let $\ga$ be a dotted composition. Then
$$S(M_{\ga})=(-1)^{\ell(\ga) + \binom{m_\ga}{2} }\sum_{\gamma\,  \trianglerighteq \,  \Rev(\ga)}M_{\gamma},$$
where $m_\alpha$ is the fermionic degree of $\alpha$.
\end{proposition}

\subsection{Fundamental quasisymmetric functions in superspace}

Just as there are two natural  extensions to superspace of Schur functions, $s_\Lambda$ and $\bar s_\Lambda$, there are two natural extension of fundamental quasisymmetric functions.  
\begin{definition}{\cite{FLP}}
  \label{def:fundamental}
The {\defstyle fundamental quasisymmetric function in superspace} $\funda$ is
$$\funda=\sum_{\gb  \preccurlyeq \ga}M_{\gb}$$
while the {\defstyle cofundamental quasisymmetric function in superspace} $\bar L_\alpha$ is
$$
\bar L_\alpha= \sum_{\gb \trianglelefteq  \ga}M_{\gb}.
$$
\end{definition}
 As stated in the introduction, 
we will only consider in this article the fundamental quasisymetric functions as the cofundamentals do not have, to the best of our knowledge, noteworthy properties.

\subsection{New characterizations of the monomial and fundamental quasisymmetric functions in superspace}
Before describing the properties of the fundamental quasisymmetric functions in superspace, we first need to provide new characterizations of the monomial and fundamental quasisymmetric functions in superspace.  These new characterizations will rely on sums over subsets of certain sets that we now introduce.

Given a dotted composition $\alpha=(\alpha_1,\dots,\alpha_l)$, we let
$$
D(\alpha)= \{\alpha_1+\eta_1,\alpha_1+\eta_1+\alpha_2+\eta_2,\dots, \alpha_1+\eta_1+\cdots+\alpha_{l-1}+\eta_{l-1}\},
$$
where we recall that $\eta_i$ was defined in \eqref{D:eta}.
For instance, given $\alpha=(2,3,\dot 1, \dot 2, 4, \dot 1, \dot 0, 2,\dot 1)$
we have $D(\alpha)=\{2,5,7,10,14,16,17,19\}$. The set $D(\alpha)$ is in the usual case (when there is no dotted entry)  naturally associated to the descent set of a permutation.  This will not be the case here.  But the set $D(\alpha)$, together with two sets that we are about to introduce, will provide a more convenient means 
to determine whether two dotted compositions are comparable in the order $\preccurlyeq$.  We now let
$$
E(\alpha)= \{ j \, |\,  \alpha_1+\eta_1+\cdots+\alpha_{i-1}+\eta_{i-1} < j < \alpha_1+\eta_1+\cdots+\alpha_{i}+\eta_i, {\rm ~for~}\alpha_i {\rm ~dotted} \}  
$$
and
$$
F(\alpha)= \{ \alpha_1+\eta_1+\cdots+\alpha_{i}+\eta_i  \, | \, \alpha_i {\rm ~is~dotted} \}  .
$$
Using again $\alpha=(2,3,\dot 1, \dot 2, 4, \dot 1, \dot 0, 2,\dot 1)$, we have $E(\alpha)=\{6,8,9,15,20\}$ and $F(\alpha)=\{ 7,10,16,17,21\}$.  Observe that
$E(\alpha)\cap D(\alpha)=\emptyset$ and that
all the entries of $F(\alpha)$ are contained in $D(\alpha)$ except possibly the largest one (this occurs when the last entry of $\alpha$ is dotted such as in our example). 
Let $F^-(\alpha)$  be equal to $F(\alpha)$ minus its largest entry if it does not belong to $D(\alpha)$.  For fixed total and fermionic degrees, there is an obvious bijection between the set of dotted compositions and the set of pairs $\bigl(D(\alpha),F(\alpha) \bigr)$, with $F^-(\alpha) \subseteq D(\alpha)$, since
$F(\alpha)$ tells us which partial sums correspond to dotted entries. 

In the non-dotted case, for $\alpha,\beta$ compositions of a fixed total degree, 
we have that $\beta \preccurlyeq \alpha \iff D(\alpha) \subseteq D(\beta)$ \cite{GR,LMW}. 
In order to extend this result to superspace, we not only need the set $F(\alpha)$, we also need the set $E(\alpha)$. Indeed, 
for $\alpha$ and $\beta$ dotted compositions
of the same total and fermionic degrees,  we have that
\begin{equation} \label{precequiv}
\beta \preccurlyeq \alpha \iff D(\alpha) \subseteq D(\beta), E(\alpha)=E(\beta) {\rm ~~and~~} F(\alpha)=F(\beta)
\end{equation}
 since the extra condition $E(\alpha)=E(\beta)$  ensures that the dotted entries are never added together  with an adjacent entry. For instance, considering $\alpha=(2,3,\dot 1,4)$ and $\beta=(2,3,1,\dot 0,4)$, which are such that $\beta \not \preccurlyeq \alpha$, we have that $D(\alpha)=\{2,5,7 \} \subseteq \{ 2,5,6,7\}=  D(\beta)$ and $F( \alpha)=F(\beta)=\{ 7\} $. But the fact that $E(\alpha)=\{6 \} \neq \emptyset =E(\beta)$ tells us that the $1$ and $\dot 0$ entries in $\beta$ were added together to form the $\dot 1$ in $\alpha$.

It is also not too hard to show that the previous equivalence can be rewritten as
\begin{equation} \label{precequiv2}
\beta \preccurlyeq \alpha \iff D(\alpha) \subseteq D(\beta), D(\beta) \cap E(\alpha)=\emptyset {\rm ~~and~~} F(\alpha)=F(\beta)
\end{equation}
since the condition $D(\beta) \cap E(\alpha)=\emptyset$ implies that the extra entries in $D(\beta)$ do not shorten the lengths of the dotted entries.

Making use of our new notation, if $\alpha \vdash (n,m)$ then we can rewrite 
$M_\alpha$ as
\begin{equation} \label{eqMvar}
M_\alpha= \sum_{\substack{ i_1 \leq i_2 \leq \cdots \leq i_{m+n} \\ i_k < i_{k+1} \iff k \in D(\alpha)   } } 
\left( \prod_{j \in F(\alpha)}  \frac{\theta_{i_j}}{x_{i_j}} \right) x_{i_1} \cdots x_{i_{n+m}}.\end{equation}

\begin{example}
  Let $\alpha=(2,\dot{3},1) \vdash(6,1)$. The corresponding monomial quasisymetric function in superspace is
  $$M_{(2,\dot{3},1)}=\sum_{i_1<i_2<i_3}\theta_{i_2}x_{i_1}^2x_{i_2}^3x_{i_3}.$$
Using $D({\alpha})=\{2,6\}$, $E({\alpha})=\{3,4,5\}$, and $F({\alpha})=\{6\}$,
the expression on the right in \eqref{eqMvar} then becomes in this case
  \begin{align*}
    M_{(2,\dot{3},1)}&=\sum_{\substack{i_1\le i_2\le i_3\le i_4\le i_5\le i_6 \le i_7\\i_k<i_{k+1}\iff k\in\{2,6\}}}\left( \prod_{j \in \{6\}}  \frac{\theta_{i_j}}{x_{i_j}} \right) x_{i_1}x_{i_2}x_{i_3}x_{i_4}x_{i_5}x_{i_6}x_{i_7}\\
&=\sum_{i_1=i_2< i_3=i_4= i_5= i_6 < i_7}\left( \frac{\theta_{i_6}}{x_{i_6}} \right) x_{i_1}^2x_{i_6}^4 x_{i_7},
  \end{align*}
which is easily seen to coincide with our previous expression for $ M_{(2,\dot{3},1)}$.   
  \end{example}

\begin{proposition} Suppose that $\alpha \vdash (n,m)$. Then
$$  
  L_\alpha=
  \sum_{\substack{ i_1 \leq i_2 \leq \cdots \leq i_{m+n} \\ i_k < i_{k+1} {\rm ~if~} k \in D(\alpha)     \\ i_k = i_{k+1} {\rm ~if~} k \in E(\alpha) } 
  } 
\left( \prod_{j \in F(\alpha)}  \frac{\theta_{i_j}}{x_{i_j}} \right) x_{i_1} \cdots x_{i_{n+m}} . $$
\end{proposition}  
\begin{proof}  The proof is similar to that of the non-supersymmetric case found in \cite{GR}. Using \eqref{eqMvar}, \eqref{precequiv} and  \eqref{precequiv2}, we get
  \begin{align}
    L_\alpha  = \sum_{\beta \preccurlyeq \alpha} M_\beta
 & =
 \sum_{\beta \preccurlyeq \alpha} \sum_{\substack{ i_1 \leq i_2 \leq \cdots \leq i_{m+n} \\ i_k < i_{k+1} \iff k \in D(\beta)  } 
  } 
\left( \prod_{j \in F(\beta)}  \frac{\theta_{i_j}}{x_{i_j}} \right) x_{i_1} \cdots x_{i_{n+m}}
 \nonumber \\ 
 &
 = \sum_{\substack{D(\alpha) \subseteq D(\beta) \\ E(\alpha)=E(\beta), F(\alpha)=F(\beta)}} \sum_{\substack{ i_1 \leq i_2 \leq \cdots \leq i_{m+n} \\ i_k < i_{k+1} \iff k \in D(\beta)  } 
  } 
\left( \prod_{j \in F(\alpha)}  \frac{\theta_{i_j}}{x_{i_j}} \right) x_{i_1} \cdots x_{i_{n+m}}
 \nonumber
 \\
 &= \sum_{\substack{D(\alpha) \subseteq Z \subseteq \{ 1,\dots,n-1 \} \setminus E(\alpha)}}
 \sum_{\substack{ i_1 \leq i_2 \leq \cdots \leq i_{m+n} \\ i_k < i_{k+1} \iff k \in Z  } 
  } 
\left( \prod_{j \in F(\alpha)}  \frac{\theta_{i_j}}{x_{i_j}} \right) x_{i_1} \cdots x_{i_{n+m}}
 \nonumber
 \\
 &=
 \sum_{\substack{ i_1 \leq i_2 \leq \cdots \leq i_{m+n} \\ i_k < i_{k+1} {\rm ~if~} k \in D(\alpha)     \\ i_k = i_{k+1} {\rm ~if~} k \in E(\alpha) } 
  } 
 \left( \prod_{j \in F(\alpha)}  \frac{\theta_{i_j}}{x_{i_j}} \right) x_{i_1} \cdots x_{i_{n+m}} .
\nonumber 
 \end{align}  
\end{proof}

\section{Coproduct of fundamentals}\label{sec:coprodOfFunds}
{
The coproduct of the fundamental quasisymmetric functions in
superspace is nearly identical to the non-super case. The difference
is in near concatenation. The {\defstyle concatenation} $\ga\cdot\gb$
of the dotted compositions $\ga=(\ga_1,\ldots,\ga_k)$ and
$\gb=(\gb_1,\ldots,\gb_{l})$ is
$(\ga_1,\ldots,\ga_k,\gb_1,\ldots,\gb_{l})$.  If not both  $\ga_k$
and $\gb_1$ are dotted, then the {\defstyle near concatenation}
$\ga\odot\gb$ of $\ga$ and $\gb$ is
$(\ga_1,\ldots,\ga_k+\gb_1,\ldots,\gb_{l}).$ If one of them is dotted, then so is the entry $\ga_k+\gb_i$. We can rewrite
Proposition~\ref{coprod} as
}
\begin{equation} \label{newcoprod}
  \Delta(M_{\ga})=\sum_{\gb\cdot\gamma=\ga}M_{\gb}\otimes M_{\gamma},
\end{equation}  
and we have the following claim for the coproduct of fundamental quasisymmetric superfunctions (we should note that the coproduct of fundamentals was also obtained in \cite{AGM} by duality with the product of noncommutative ribbon Schur functions in superspace)
\omitt{
\Susanna{\begin{claim} \label{claimcoprod}
   Let $\ga$ be a dotted composition. Then
  $$\Delta(\funda)=\sum_{\substack{\gb\cdot\gamma=\ga\text{ or
      }\\\gb\odot\gamma=\ga}}\fund{\gb}\otimes \fund{\gamma}.$$
\lltodo{We should have a formula for $\Delta(\bar L_\alpha)$}
 \end{claim}
}
}

\begin{proposition} Let $\alpha$ be a dotted composition.  We have
  \label{claimcoprod}
  \begin{equation}
    \label{eq:fundcoprod}
  \Delta(L_{\alpha})=\sum_{\substack{\eta\cdot\gamma=\alpha\text{ or
      }\\ \eta\odot\gamma=\alpha}} L_{\eta}\otimes L_{\gamma}.
  \end{equation}
\end{proposition}

\begin{proof} We use the interpretation \eqref{dupli} of the coproduct as the duplication of alphabets.  We have
  $$
  \Delta(L_{\alpha})= L_\alpha(x,y; \theta,\phi) = \sum_{s=0}^{n+m}
  \sum_{\substack{ i_1 \leq i_2 \leq \cdots \leq i_s \\ i_{s+1} \leq i_{s+2} \leq \cdots \leq i_{n+m} \\ i_k < i_{k+1} {\rm ~if~}  k \in D(\alpha) \setminus \{ s\}\\ i_k = i_{k+1} {\rm ~if~}  k \in E(\alpha)  \setminus \{ s\} } } 
  \left( \prod_{j \in F(\alpha); j \leq s}\frac{\theta_{i_j}}{x_{i_j}} \right)
x_{i_1} \cdots x_{i_s}
\left( \prod_{\ell \in F(\alpha); \ell > s}\frac{\phi_{i_\ell}}{y_{i_\ell}} \right)
y_{i_{s+1}} \cdots y_{i_{m+n}}.
$$
Is is then easy to see that if $s \in D(\alpha) \cup \{ 0,n+m\}$, then the summand corresponds to $L_\beta(x;\theta) L_\gamma(y;\phi) $ with $\beta$ and $\gamma$ such that $\beta \cdot \gamma=\alpha$ while the remaining values of $s$ correspond to $L_\beta(x;\theta) L_\gamma(y;\phi) $ with $\beta$ and $\gamma$ such that $\beta \odot \gamma=\alpha$.
  \end{proof}

\section{Products of fundamentals}
\label{sec:prodOfFunds}
In \cite{GR} and \cite{LMW}, shuffles are defined on
permutations/linear extensions and used to describe the product of
quasisymmetric functions. In \cite{FLP}, overlapping shuffles are used
to describe the product of monomial symmetric functions in superspace.
For fundamental quasisymmetic functions in superspace, we will
define \newShuf shuffles. The product of two fundamental
quasisymmetric functions in superspace, $\funda$ and $\fund{\gb}$, is
the sum over all \newShuf shuffles of two permutations $w_{\ga}$ and
$w_{\gb}$, where the descent sets of $w_{\ga}$ and $w_{\gb}$ are the
sets of partial sums of $\ga$ and $\gb$ respectively.

Before we delve into the definition of fundamental shuffles, we review overlapping shuffles, which will appear in the proof of Proposition~\ref{prop:fundProd}. Let $\ga$ and $\gb$ be two dotted compositions of length $\ell$ and $\ell'$ respectively. We make a $\ell'$
by $\ell$ grid and label the rows by $\gb$ and the columns by $\ga$.
If both $\ga_q$ and $\gb_p$ are dotted, then place a dot in the cell
in row $p$ and column $q$. The path $P$ in the $(x,y)$ plane
from $(0,0)$ to $(\ell,\ell')$ with steps $(0,1)$, $(1,0)$, and
$(1,1)$ is similar to the paths defined in \cite[Section 3.3.1]{LMW}.
In our case, paths are not allowed to step diagonally over cells where
both $\ga_q$ and $\gb_p$ are dotted.  Each path corresponds to a dotted composition $\gamma=\gamma(P)$. If the $h^{\text{th}}$ step of $P$
is horizontal over column $q$, then $\gamma_h=\ga_q$. If it is a  vertical step over row $p$, then $\gamma_h=\gb_p$. Finally, if it is a diagonal step over row $p$ and column $q$, then$\gamma_h=\ga_q+\gb_p$.  In all cases, $\gamma_h$ is dotted if either $\ga_q$ or $\gb_p$ is. We call the set of all
paths which can be obtained from $\ga$ and $\gb$ in this manner {\defstyle
  the set of $(\ga,\gb)$ overlapping shuffles}.

\omitt{\begin{color}{blue}We need
    something similar, but not the whole theory of $P$-partitions. I
    don't see how to extend the definition of linear extensions to the
    super case, so we have no theorem similar to Lemma 3.3.23 in
    \cite{LMW}.} We need dotted permutations, which are analogous to
$P$-partitions where the underlying poset is a chain.

\begin{definition}
A {\defstyle dotted permutation} is a permutation $w=\pl w_1,\ldots,w_n\pr$ of a
multisubset of $\Z$, where some of the entries may be dotted and where
the non-dotted entries form a subset of $\Z$; in particular, no dotted entry may be repeated. Write $w^{\circ}$ for
$w$ with the dotted entries deleted. The {\defstyle size} of
$w=\pl w_1,\ldots,w_n\pr$ is $\size(w)=n$.  If $w_i^{\circ}>w_{i+1}^{\circ}$
or if $w^{\circ}_i$ is followed by a dotted entry in $w$, then $i$ is
a {\defstyle descent} for $w$. The indices for the descents of $w$ come from
the indices of $w^{\circ}.$
\end{definition}
\begin{example}
Let $w=\pl6,7,8,\dot{6},\dot{3},9,10,3,\dot{1},2,4\pr$. Then $w^{\circ}=\pl6,7,8,9,10,3,2,4\pr$ and the descents of $w$ are $\{3, 5,6\}$.
\end{example}

We define a map $\comp(w)$ from dotted
permutations to dotted compositions. Restricted to non-dotted permutations, the map is the usual one. We need to be a little careful when inserting the dotted parts.

\begin{definition}
Let $w=\pl w_1,\ldots,w_n\pr$ be a dotted permutation with descents
$d_1<d_2<\cdots <d_k$ and such that $w^{\circ}$ is a permutation of a set
with $N$ elements. The non-dotted parts of $\comp(w)$ are
$(d_1,d_2-d_1,\dots,d_{k}-d_{k-1},N-d_k)$ in order. If the dotted entry
$w_i$ follows $w^{\circ}_{d_j}$ in $w$, then insert $w_i$ after
$d_j-d_{j-1}$ in $\comp(w)$. If a dotted entry follows another dotted
entry in $w$, then it should also follow it in $\comp(w)$.
\end{definition}

\begin{example}
Let $w=\pl6,7,8,\dot{6},\dot{3},9,10,3,\dot{1},2,4\pr$. Then
$\comp(w)=(3,\dot{6},\dot{3},2,1,\dot{1},2)$. Let
$w'=\pl6,7,9,\dot{6},\dot{3},8,10,1,\dot{1},2,3\pr$. Then $\comp(w')$ is
also equal to $(3,\dot{6},\dot{3},2,1,\dot{1},2)$.

\end{example}

Let $\ga$ be a dotted composition. Then we say the dotted permutation
$w_{\ga}$ {\defstyle represents} $\ga$ on the set $S$ if
$\comp(w_{\ga})=\ga$ and $w_{\ga}^{\circ}$ is a permutation of $S$.

We will again need paths to define shuffles, as in \cite{GR,LMW,FLP}, but
extend their definition.

Let $\ga$ and $\gb$ be two dotted compositions. Let $w_{\ga}$
represent $\ga$ on $S$ and $w_{\gb}$ represent $\gb$ on $T$, where
$s<t$ for all $s\in S$ and $t\in T$.
A {\defstyle $(\ga,\gb)$-grid} is a $\size(w_\gb)\times\size(w_\ga)$ grid in
the first quadrant where each row $i$ of cells is labeled by the entries of $w_\gb$ and the
columns by the entries of $w_{\ga}$. Cell $(i,j)$ is the cell in row $i$,
column $j$, with corners $(i-1,j-1)$, $(i-1,j)$, $(j,i-1)$, and
$(i,j)$. Row numbering starts with the bottom row of the grid, and column numbering with the leftmost column. A {\defstyle \newShuf path} in the grid is from
(0,0) to $(\size(w_\ga),\size(w_\gb))$.  We define a path by steps. Suppose the path has reached the southwest corner of cell $(i,j)$. There are four possible
steps for a \newShuf path:
\begin{enumerate}
\item \label{step1}horizontal (1,0);
\item \label{step2}vertical (0,1);
\item \label{step3}if $(w_{\ga})_j$ is a dotted entry and if the
  labels on successive rows of the grid are increasing and non-dotted:
  $(w_{\gb})_i<(w_{\gb})_{i+1}<\cdots<(w_{\gb})_{i+k-1}$, then the
  diagonal step $(i,j)\mapsto (i+k,j+1)$ is possible; and
\item \label{step4}if a $(w_{\gb})_i$ is a dotted entry and if the
  labels on successive
  columns of the grid are increasing and non-dotted:
  $(w_{\ga})_j<(w_{\ga})_{j+1}<\cdots<(w_{\ga})_{j+k-1}$, then the
  diagonal step $(i,j)\mapsto (i+1,j+k)$ is possible.
\end{enumerate}

We allow diagonal steps over more than one cell in the grid. In \cite{GR,LMW,FLP}, diagonal steps were over single cells. 
The set of paths we obtain does not depend on which permutations $w_{\ga}$ and $w_{\gb}$ we choose to represent $\ga$ and $\gb$.

Each path $P$ in a $(\ga,\gb)$-grid corresponds to a dotted
permutation $\Pi(P)$.  We now describe the correspondence. Let $P$ be a path in a $(\ga,\gb)$-grid. The bijection proceeds in steps, with $\Pi(P)$ initially an empty permutation and the steps of $P$ processed from lower left to upper right.  The number of entries in the dotted permutation $\Pi(P)$ 
will be the number of steps in $P$. For a horizontal step of $P$ in
column $j$, append the label $(w_{\ga})_{j}$ to the permutation. For a
vertical step in $P$ in row $i$ add the label $(w_{\gb})_{i}$ to the
permutation. Suppose the path takes a diagonal step of type (\ref{step3}),
where $\dot{a}$ labels column $j$. Then append $\dot{(a+k)}$ to the
permutation. Similarly, for a step of type (\ref{step4}), where
$\dot{b}$ labels row $i$, add $\dot{(b+k)}$ to the permutation. Please see Example~\ref{ex:bigEx} and Figure~\ref{fig:bijectionPi}.

The {\defstyle \newShuf shuffle set} of $\ga$ and $\gb$ is the set of dotted compositions
produced by $\comp(w)$ for dotted permutations $w$ from paths in the
$(\ga,\gb)$-grid. In other words, 
  the \newShuf shuffle set of $\ga$ and $\gb$ is the set
  $$
\left \{ \comp(w) \, |\, w=\Pi(P), \text{$P$ is a \newShuf path in the $(\ga,\gb)$-grid} \right \} .
  $$

\begin{example}
\label{ex:bijectionPi}
Let $\ga$ be the dotted composition $(1,1,\dot{3})$ and $\gb$ be $(1,\dot{0})$. Let $w_{\ga}=[4,3,\dot{3},2,6]$ and $w_{\gb}=[8,\dot{0},7]$ be the dotted permutations we choose to represent $\ga$ and $\gb$, respectively. In Figure~\ref{fig:bijectionPi}, we draw a path $P$ in the $(\ga,\gb)$-grid. There are six entries in the dotted permutation $\Pi(P)$, because $P$ has six steps. The first two steps are horizontal, so we append $(w_{\ga})_1$ and  $(w_{\ga})_2$, the vertical step contributes  $(w_{\gb})_1$, and the fourth step is horizontal and adds  $(w_{\ga})_3$. The fifth step of $P$ is a diagonal step of type \eqref{step4}, where $(w_{\gb})_2=\dot{0}$ labels row $2$ and $k=2$ is the number of columns whose labels are increasing. We add $\dot{2}$ to the permutation. Finally, we add $7$, so that $\Pi(P)=[4,3,8,\dot{3},\dot{2},7]$. The dotted composition corresponding to $P$ is $\Gamma(P)=\comp(\Pi(P))=(1,2,\dot{3},\dot{2},1)$.
\end{example}
\begin{figure}[h]
    \centering

\begin{tikzpicture}
\def\labelHt{-.45}
\def\permHt{1.6}
\begin{scope}
\node at (.5,\labelHt){$4$};
\node at (1.5,\labelHt){$3$};
\node at (2.5,\labelHt){$\dot{3}$};
\node at (3.5,\labelHt){$2$};
\node at (4.5,\labelHt){$6$};
\node at (-.5,.5){$8$};
\node at (-.5,1.5){$\dot{0}$};
\node at (-.5,2.5){$7$};
\draw[help lines] (0,0) grid (5,3);
\draw[red,line width=1.5pt](0,0)-- ++(1,0)-- ++(1,0)-- ++(0,1)-- ++(1,0)--++(2,1)--++(0,1);
\filldraw[blue] (2.5,1.5) circle (2pt);
\end{scope}
\end{tikzpicture}
 \caption{The grid and path described in Example~\ref{ex:bijectionPi}}
    \label{fig:bijectionPi}
\end{figure}
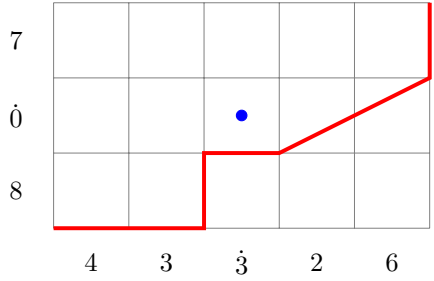

Keeping the notation used for overlapping shuffles, we will denote by
$\Gamma(P)$ the dotted composition corresponding to the \newShuf path $P$, that is,
$\Gamma(P)=\comp(\Pi(P))$.

Each \newShuf shuffle also has a sign. When both $(w_{\ga})_j$ and
$(w_{\gb})_i$ are dotted, put a dot in the square in the
$(i-1)^{\textrm{th}}$ row and $(j-1)^{\textrm{th}}$ column. The
{\defstyle sign} of the path/dotted permutation/dotted composition is
$(-1)^{\textrm{number of dots below the path}}$.

\begin{center}
\begin{figure}
  \input{map4Mult}
  \caption{Multiplication of $\fund{(\dot{1},2)}$ and $\fund{(\dot{2},1)}$}
  \label{fig:mult}
  \end{figure}
\end{center}

\begin{proposition}\label{prop:fundProd}
\begin{equation}
\label{eq:fundProd}
\fund{\ga}\fund{\gb}=\sum_{\textrm{\newShuf shuffles $\gamma$ of $\ga$ and
      $\gb$}}\sign(\gamma)\fund{\gamma}.
\end{equation}
\end{proposition}

\begin{example}
\label{ex:fundProd}
 The goal of this example and Figure~\ref{fig:mult} is to illustrate  Proposition~\ref{prop:fundProd} and its proof. In the first column of Figure~\ref{fig:mult} we have all \newShuf paths $Q$ in $(\ga,\gb)$-grid for
$\ga=(\dot{1},2)$ and $\gb=(\dot{2},1)$. The dotted permutation
$\pl\dot{1},1,2\pr$ was used to represent $\ga$ and the dotted permutation
$\pl\dot{2},3\pr$ to represent $\gb$. In the second column, to the right of
each path, are the dotted permutation corresponding to the path, its descent set, and the dotted composition $\gamma=\Gamma(Q)$ corresponding to the path. In the third column, we have all the dotted compositions $\gamma'$ such that $\Gamma(Q)=\gamma\compgeq\gamma'$.  The proof of Proposition~\ref{prop:fundProd} depends on a bijection $\phi$. In the rightmost column, we have $\phi(\gamma')$ for each $\gamma'$ in the third column.


  According to Proposition~\ref{prop:fundProd},
  \begin{eqnarray*}
    \fund{(\dot{1},2)}\fund{(\dot{2},1)}&=&\fund{(\dot{1},2,\dot{2},1)} +\fund{(\dot{1},1,\dot{2},2)} +\fund{(\dot{1},1,\dot{2},1,1)}+\fund{(\dot{1},1,\dot{3},1)}+\fund{(\dot{1},\dot{2},3)}+\fund{(\dot{1},\dot{2},2,1)}+\fund{(\dot{1},\dot{2},1,2)}+\fund{(\dot{1},\dot{3},2)}\\
    &&+\fund{(\dot{1},\dot{3},1,1)}+\fund{(\dot{1},\dot{4},1)}-\fund{(\dot{2},\dot{1},3)}-\fund{(\dot{2},\dot{1},2,1)}-\fund{(\dot{2},\dot{1},1,2)}-\fund{(\dot{2},1,\dot{1},2)}-\fund{(\dot{2},\dot{2},2)} .
    \end{eqnarray*}
  \end{example}

\begin{proof}[Proof of Proposition~\ref{prop:fundProd}]
  By definition of fundamental quasisymmetric functions in superspace (Definition~\ref{def:fundamental}) and Proposition 5.5 of \cite{FLP},
  the proposition amounts to
  \begin{equation}
  \label{eq:prodproof}
\sum_{\alpha' \preccurlyeq \alpha} \sum_{\beta' \preccurlyeq \beta} \sum_{P} \sign(P)  M_{\Gamma(P)}  = \sum_{\textrm{\newShuf shuffles $\gamma$ of $\ga$ and
      $\gb$}}\sign(\gamma)\sum_{\gamma' \preccurlyeq \gamma} M_{\gamma'},
  \end{equation}  
where the sum on the left is over all $P$ that are overlapping  $(\ga',\gb')$ shuffles.
In view of the definition of \newShuf $(\ga,\gb)$ shuffles, the proposition will hold if we can obtain a bijection $ \phi$
between
the set $S_1$ of indices on the right hand side of \eqref{eq:prodproof}
$$S_1:=\{(Q,\gamma'):\gamma' \preccurlyeq \Gamma(Q),Q\text{ is a
   \newShuf $(\ga,\gb)$ shuffle}\}$$
  and the set $S_2$ of indices on the left hand side of \eqref{eq:prodproof}
  $$S_2:=\{(P,(\ga',\gb')):\ga' \preccurlyeq \ga ,\gb' \preccurlyeq \gb,P\text{
    is a $(\ga',\gb')$ overlapping shuffle}\}$$
  with the property that $\phi(Q,\gamma')=(P,(\ga',\gb'))$ is such that
  $\Gamma(P)=\gamma'$ and
  $\sign(P)=\sign(Q)$. Reminder: paths of fundamental shuffles are displayed on a grid where the rows and columns are labeled by permutations that represent dotted compositions, whereas paths of overlapping shuffles are displayed on a grid where the rows and columns are labeled by the compositions themselves. 
  
Consider an element of $S_1$: we are given $Q$, a path
on an $(\ga,\gb)$-grid, and $\gamma'\preccurlyeq\gamma=\Gamma(Q)$.  We
deform $Q$ into $P$ based on $\gamma'$ and build $\ga'$ and $\gb'$
along the way. Both $\ga'$ and $\gb'$ are initially empty and we add parts to them at each step. At step $h$, we consider $\gamma_h$, the $h^{\text{th}}$ part
$\gamma=\Gamma(Q)$. Assume we have constructed the initial parts of
$\ga'$ and the initial parts of $\gb'$ based on
$\gamma_1,\ldots,\gamma_{h-1}$.

We begin with the more difficult case: 
$\gamma_h$ is not dotted. Since $\gamma'$ refines $\gamma$ and $w$ represents $\gamma$, there are
integers $i$, $j$, and $x$ such that
$\gamma_h=\gamma'_i+\cdots+\gamma'_j$ and $\gamma_h$ corresponds to
$w_{x+1}<w_{x+2}<\cdots<w_{x+\gamma_h}$,
where $w=\Pi(Q)$ and the following holds.
\begin{itemize}
\item either $w_{x+1}$ is dotted or $w_x>w_{x+1}$, and
\item either $w_{x+\gamma_k+1}$ is
  dotted or $w_{x+\gamma_h}>w_{x+\gamma_h+1}$ or $x+\gamma_h=\size(w)$.
\end{itemize}

We must determine how to allocate $\gamma'_i,\gamma'_{i+1},\ldots,\gamma'_j$ between $\alpha'$ and $\beta'$.

Since the undotted entries in $w_{\ga}$ are all less than the undotted
entries in $w_{\gb}$, there is an integer $a$, $0\leq
a\leq \gamma_h$ such that
$w_{x+1},\ldots,w_{x+a}$ are all entries from $w_{\ga}$ and
$w_{x+a+1},\ldots,w_{x+\gamma_h}$ are all entries from
$w_{\gb}$.

Suppose $\gamma'_1+\cdots+\gamma'_k\leq a$ and either
$\gamma'_1+\cdots+\gamma'_{k+1}> a$ or $k=j$. Add the parts
$\gamma'_i,\ldots,\gamma'_k,a-(\gamma'_i+\cdots+\gamma'_k)$ to the
$\ga'$ we are building and add
$\gamma'_i+\cdots+\gamma'_{k+1}-a,\gamma'_{k+2},\ldots,\gamma'_j$ to
the $\gb'$ we are building.

$$
\text{entries of $w$:\quad}\overbrace{\underbrace{w_{x+1}\ldots}_{\gamma'_i}\underbrace{w_{x+\gamma'_i+1}\ldots}_{\gamma'_{i+1}}\dots\underbrace{\ldots w_{x+a}}_{a-(\gamma'_i+\cdots+\gamma'_k)}}^{\text{entries of $w$ from $w_{\ga}$}}\overbrace{\underbrace{w_{x+a+1}\ldots}_{\gamma'_i+\cdots+\gamma'_{k+1}-a}\underbrace{w_{x+\gamma'_i+\cdots+\gamma'_{k+1}+1}\ldots}_{\gamma'_{k+2}}\ldots\underbrace{\ldots w_{x+\gamma_h}}_{\gamma'_j}}^{\text{entries of $w$ from $w_{\gb}$}}
$$

The path $P$ will be horizontal over the
columns labeled with $\gamma'_i,\ldots,\gamma'_k$, diagonal over the
cell whose column label is $a-(\gamma'_i+\cdots+\gamma'_k)$ and row
label is $\gamma'_i+\cdots+\gamma'_{k+1}-a$ and vertical over the rows
labeled $\gamma'_{k+2},\ldots,\gamma'_j$.  Note that if
    $a-(\gamma'_i+\cdots+\gamma'_k)$ or $\gamma'_i+\cdots+\gamma'_{k+1}-a$ is equal to zero, then the path will be vertical or horizontal respectively instead of
diagonal over the corresponding columns and rows.
See Figure~\ref{fig:noDotPortion} to see the effect of $\phi$ on the portion of $Q$ corresponding to the undotted part $\gamma_h.$ By construction, $\ga'$ refines $\ga$ and $\gb'$ refines $\gb$.

\begin{center}
  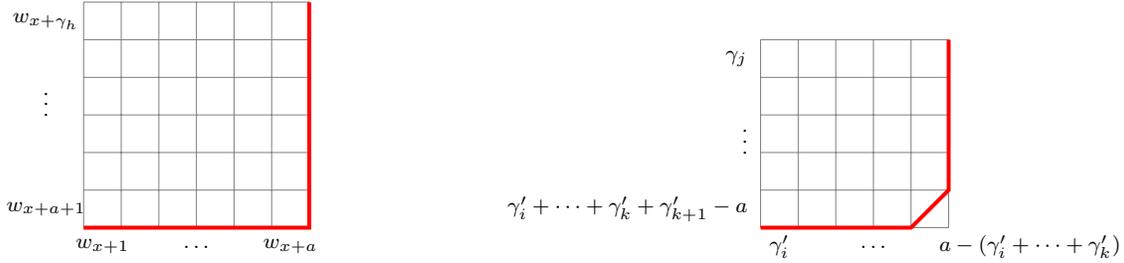
\begin{figure}[ht]
\begin{tikzpicture}[scale=.5,font=\footnotesize]
\begin{scope}[shift={(0,0)}]
\node at (.5,-.5){$w_{x+1}$};
\node at (3.,-.5){$\hdots$};
\node at (5.5,-.5){$w_{x+a}$};

\node at (-1,.5){$w_{x+a+1}$};
\node at (-1,3.5){$\vdots$};
\node at (-1,5.5){$w_{x+\gamma_h}$};

\draw[help lines] (0,0) grid (6,6);
\draw[red,line width=1.5pt](0,0)-- ++(1,0)-- ++(1,0)-- ++(1,0)-- ++(1,0)-- ++(1,0)-- ++(1,0)-- ++(0,1)-- ++(0,1)-- ++(0,1)-- ++(0,1)-- ++(0,1)-- ++(0,1)   ;
\end{scope}
\begin{scope}[shift={(18,0)}]
\node at (.5,-.5){$\gamma'_i$};
\node at (3.,-.5){$\hdots$};
\node at (4.5,-.5)[right]{$a-(\gamma'_i+\cdots+\gamma'_k)$};

\node at (-.1,.5)[left]{$\gamma'_i+\cdots+\gamma'_k+\gamma'_{k+1}-a$};
\node at (-.1,2.5)[left]{$\vdots$};
\node at (-.1,4.5)[left]{$\gamma_j$};

\draw[help lines] (0,0) grid (5,5);
\draw[red,line width=1.5pt](0,0)-- ++(1,0)-- ++(1,0)-- ++(1,0)-- ++(1,0)-- ++(1,1)-- ++(0,1)-- ++(0,1)-- ++(0,1)-- ++(0,1);
\end{scope}

\end{tikzpicture}

  \label{fig:noDotPortion}
  \caption{The map $\phi$ deforms the portion of $Q$ on the left into
    a portion of $P$ and parts of $\ga'$ and $\gb'$ on the right.}
      \end{figure}
  \end{center}

Now suppose $\gamma_h$ is dotted. If it is produced by a step of
type \eqref{step1}, add $\gamma_h$ to $\ga'$ and $P$ will be horizontal for its next step. Likewise, if is
produced by a step of type \eqref{step2}, add $\gamma_h$ to
$\gb'$ and a vertical step to $P$. Otherwise, it is produced by a step of type \eqref{step3}
or \eqref{step4}. Assume it is of type
\eqref{step3}. Then add $(w_\ga)_i$ (dotted) to $\ga'$ as a part and
add $k$ to $\gb'$, where the notation is as in the description of
shuffles. The path $P$ will step diagonally over the cell whose column
label is $(w_\ga)_i$ and whose row label is $k$. Again, the
construction ensures that $\ga'$ and $\gb'$ refine $\ga$ and
$\gb$. See Figure~\ref{fig:dotPortion}.

It is not difficult to see that these steps can be reversed and $\phi$ is a bijection.
\begin{center}
  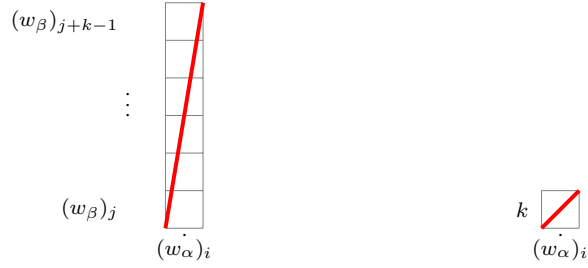
\begin{figure}[ht]
\begin{tikzpicture}[scale=.5,font=\footnotesize]
\begin{scope}[shift={(0,0)}]
\node at (.5,-.5){$\dot{(w_{\ga})_i}$};

\node at (-1,.5)[left]{$(w_{\gb})_j$};
\node at (-1,3.5){$\vdots$};
\node at (-1,5.5)[left]{$(w_{\gb})_{j+k-1}$};

\draw[help lines] (0,0) grid (1,6);
\draw[red,line width=1.5pt](0,0)-- ++(1,6);
\end{scope}

\begin{scope}[shift={(10,0)}]
\node at (.5,-.5){$\dot{(w_{\ga})_i}$};

\node at (-.1,.5)[left]{$k$};

\draw[help lines] (0,0) grid (1,1);
\draw[red,line width=1.5pt](0,0)-- ++(1,1);
\end{scope}

\end{tikzpicture}

  \label{fig:dotPortion}
  \caption{The map $\phi$ deforms the portion of $Q$ on the left into a portion of $P$ and parts of $\ga'$ and $\gb'$ on the right, in the dotted case.}
      \end{figure}
  \end{center}

\end{proof}

\begin{example} This is a continuation of Example~\ref{ex:fundProd}. 
  In Figure~\ref{fig:mult}, the dotted compositions $\gamma'$ in the
  third column, together with the corresponding path $Q$ in the first
  column are matched by $\phi$ to the respective path in the fourth
  column. The $(\ga',\gb')$ pair which refines $(\ga,\gb)$ is given by
  the labels on the grid containing the path $P$. In this example,
  $\ga=(\dot{2},1)$ and $\gb=(\dot{1},2)$.
  \end{example}

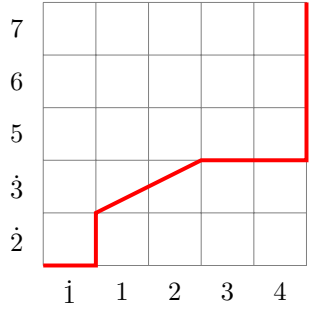
\begin{figure}[h]
\centering
\begin{tikzpicture}[scale=.7]
\def\xsp{1}
\def\ysp{1}
\def\xbegin{.5}
\def\ybegin{.5}
 \draw[help lines] (0,0) grid (5,5);
\node at (\xbegin,-\ybegin) {$\dot{1}$};
\node at (\xbegin+\xsp,-\ybegin) {$1$};
\node at (\xbegin+2*\xsp,-\ybegin) {$2$};
\node at (\xbegin+3*\xsp,-\ybegin) {$3$};
\node at (\xbegin+4*\xsp,-\ybegin) {$4$};

\node at (-\xbegin,\ybegin) {$\dot{2}$};
\node at (-\xbegin,\ybegin+1*\ysp) {$\dot{3}$};
\node at (-\xbegin,\ybegin+2*\ysp) {$5$};
\node at (-\xbegin,\ybegin+3*\ysp) {$6$};
\node at (-\xbegin,\ybegin+4*\ysp) {$7$};

\draw[red,line width=1.5pt](0,0)-- ++(1,0)-- ++(0,1)-- ++(2,1)-- ++(2,0)-- ++(0,3);
\end{tikzpicture}
\caption{The path $Q$ representing the fundamental shuffle discussed in Example~\ref{ex:bigEx}. It appears in the expansion of $L_{(\dot{1},4)}$ and $L_{(\dot{2},\dot{3},3)}$.}
\label{fig:Q}
\end{figure}

\begin{example}\label{ex:bigEx}
In this example, we let $\ga=(\dot{1},4)$, $\gb=(\dot{2},\dot{3},3)$,
$w_{\ga}=\pl\dot{1},1,2,3,4\pr$, and
$w_{\gb}=\pl\dot{2},\dot{3},5,6,7\pr$. We let $Q$ be the fundamental shuffle whose
path is given Figure~\ref{fig:Q}. We have
$w=\Pi(Q)=\pl\dot{1},\dot{2},\dot{5},3,4,5,6,7\pr$ and
$\gamma=\Gamma(Q)=(\dot{1},\dot{2},\dot{5},5)$. We consider two examples of dotted compositions $\gamma'$ such that
$\gamma'\preccurlyeq\gamma$ and calculate their image under the map described in the proof of Proposition~\ref{prop:fundProd}.
\begin{enumerate}
\item Let $\gamma'=\gamma=(\dot{1},\dot{2},\dot{5},5)$. Initially $\ga'=\gb'=\emptyset$.
\begin{enumerate}
\item $\gamma_1=\dot{1}$ and $Q$'s first step is horizontal. Add $\dot{1}$ to $\ga'$. $P$'s first step is horizontal.
\item $\gamma_2=\dot{2}$ and $Q$'s second step is vertical. Add $\dot{2}$ to $\gb'$. $P$'s second step is vertical.
\item Step 3 is a diagonal step of type \ref{step4}. We add $\dot{3}$ to $\gb'$ and $2$ to $\ga'$. $P$'s third step is diagonal.
\item We have $\gamma_4=5$, $x=3$, and $w_{x+1},w_{x+2},\ldots,w_{x+\gamma_4}=3,4,5,6,7$ and $i=j=4$. Since $w_4$ and $w_5$ come from $w_{\ga}$, $a=2$ and we add $2$ to $\ga'$ and $3$ to $\gb'$. $P$'s final step is diagonal. See Figure~\ref{fig:QSteps}.
\end{enumerate}
The overlapping shuffle which corresponds to $P$ and $(\ga',\gb')$ is $\gamma'$.
\item Let $\gamma'=(\dot{1},\dot{2},\dot{5},1,2,1,1)$. Initially $\ga'=\gb'=\emptyset$.

\begin{enumerate}
\item $\gamma_1=\dot{1}$ and $Q$'s first step is horizontal. Add $\dot{1}$ to $\ga'$. $P$'s first step is horizontal.
\item $\gamma_2=\dot{2}$ and $Q$'s second step is vertical. Add $\dot{2}$ to $\gb'$. $P$'s second step is vertical.
\item Step 3 is a diagonal step of type \ref{step4}. We add $\dot{3}$ to $\gb'$ and $2$ to $\ga'$. $P$'s third step is diagonal.
\item We have $h=4$, $i=4$, and $j=7$ in the fourth step. Again we have $a=2$. We add $\gamma'_4,a-\gamma'_4$ to $\ga'$ and $\gamma'_4+\gamma'_5-a,\gamma'_6,\gamma'_7$ to $\gb'$. $P$ will be horizontal over the column labelled by $\gamma'_4$ ($1$), diagonal over the column labelled $by a-\gamma'_4$ and row labelled by $\gamma'_4+\gamma'_5-a$, and vertical over the rows labelled by $\gamma'_6$ and $\gamma'_7. $ See Figure~\ref{fig:QStepsSecond}.  
\end{enumerate}
Again, the overlapping shuffle which corresponds to $P$ and $(\ga',\gb')$ is $\gamma'$.
\end{enumerate}
\end{example}

\begin{figure}[h]
\centering
\begin{tikzpicture}[scale=.7]
\def\xsp{1}
\def\ysp{1}
\def\xbegin{.3}
\def\ybegin{.5}

\begin{scope}
 \draw[help lines] (0,0) grid (1,0);
\node at (\xbegin,-\ybegin) {$\dot{1}$};


\draw[red,line width=1.5pt](0,0)-- ++(1,0);
\end{scope}

\begin{scope}[shift={(2,0)}]
 \draw[help lines] (0,0) grid (1,1);
\node at (\xbegin,-\ybegin) {$\dot{1}$};

\node at (-\xbegin,\ybegin) {$\dot{2}$};

\draw[red,line width=1.5pt](0,0)-- ++(1,0)-- ++(0,1);
\end{scope}

\begin{scope}[shift={(4,0)}]
 \draw[help lines] (0,0) grid (2,2);
\node at (\xbegin,-\ybegin) {$\dot{1}$};
\node at (\xbegin+\xsp,-\ybegin) {$2$};

\node at (-\xbegin,\ybegin) {$\dot{2}$};
\node at (-\xbegin,\ybegin+\ysp) {$\dot{3}$};

\draw[red,line width=1.5pt](0,0)-- ++(1,0)-- ++(0,1)-- ++(1,1);
\end{scope}

\begin{scope}[shift={(7,0)}]
 \draw[help lines] (0,0) grid (3,3);
\node at (\xbegin,-\ybegin) {$\dot{1}$};
\node at (\xbegin+\xsp,-\ybegin) {$2$};
\node at (\xbegin+2*\xsp,-\ybegin) {$2$};

\node at (-\xbegin,\ybegin) {$\dot{2}$};
\node at (-\xbegin,\ybegin+\ysp) {$\dot{3}$};
\node at (-\xbegin,\ybegin+2*\ysp) {$3$};

\draw[red,line width=1.5pt](0,0)-- ++(1,0)-- ++(0,1)-- ++(1,1)-- ++(1,1);-- ++(0,2);
\end{scope}
\end{tikzpicture}
\caption{The construction of the path $P$, $\ga'$, $\gb'$ which corresponds to $Q$, $\ga$, $\gb$, $\gamma$, and $\gamma'=(\dot{1},\dot{2},\dot{5},5)$ in Example~\ref{ex:bigEx}.} 
\label{fig:QSteps}
\end{figure}

\begin{figure}[h]
\centering
\begin{tikzpicture}[scale=.7]
\def\xsp{1}
\def\ysp{1}
\def\xbegin{.3}
\def\ybegin{.5}

\begin{scope} 
 \draw[help lines] (0,0) grid (1,0);
\node at (\xbegin,-\ybegin) {$\dot{1}$};


\draw[red,line width=1.5pt](0,0)-- ++(1,0);
\end{scope}

\begin{scope}[shift={(2,0)}] 
 \draw[help lines] (0,0) grid (1,1);
\node at (\xbegin,-\ybegin) {$\dot{1}$};

\node at (-\xbegin,\ybegin) {$\dot{2}$};

\draw[red,line width=1.5pt](0,0)-- ++(1,0)-- ++(0,1);
\end{scope}

\begin{scope}[shift={(4,0)}] 
 \draw[help lines] (0,0) grid (2,2);
\node at (\xbegin,-\ybegin) {$\dot{1}$};
\node at (\xbegin+\xsp,-\ybegin) {$2$};

\node at (-\xbegin,\ybegin) {$\dot{2}$};
\node at (-\xbegin,\ybegin+\ysp) {$\dot{3}$};

\draw[red,line width=1.5pt](0,0)-- ++(1,0)-- ++(0,1)-- ++(1,1);
\end{scope}

\begin{scope}[shift={(7,0)}] 
 \draw[help lines] (0,0) grid (4,5);
\node at (\xbegin,-\ybegin) {$\dot{1}$};
\node at (\xbegin+\xsp,-\ybegin) {$2$};
\node at (\xbegin+2*\xsp,-\ybegin) {$1$};
\node at (\xbegin+3*\xsp,-\ybegin) {$1$};

\node at (-\xbegin,\ybegin) {$\dot{2}$};
\node at (-\xbegin,\ybegin+\ysp) {$\dot{3}$};
\node at (-\xbegin,\ybegin+2*\ysp) {$1$};
\node at (-\xbegin,\ybegin+3*\ysp) {$1$};
\node at (-\xbegin,\ybegin+4*\ysp) {$1$};

\draw[red,line width=1.5pt](0,0)-- ++(1,0)-- ++(0,1)-- ++(1,1)-- ++(1,0)-- ++(1,1)-- ++(0,2);
\end{scope}

\end{tikzpicture}
\caption{The construction of the path $P$, $\ga'$, $\gb'$ which correspond to $Q$, $\ga$, $\gb$, $\gamma$, and $\gamma'=(\dot{1},\dot{2},\dot{5},1,2,1,1)$ in Example~\ref{ex:bigEx}.} 
\label{fig:QStepsSecond}
\end{figure}

\section{Antipode}
\label{sec:antipodes}
In the non-dotted case, the antipode of $L_\alpha$ is simply 
$L_{\alpha^t}$ (up to a sign), where
$\alpha^t$ is the composition obtained by transposing the ribbon associated to $\alpha$. As we will see, the antipode of $L_\alpha$ is not that simple in the dotted case.  Fortunately, we can define two operations on   ${\rm sQSym}$, whose compatibility with the antipode allows  to decompose the coproduct into simple terms (see \eqref{eqdecomp} and Proposition~\ref{propcolumn}).  

Recall that for $\alpha=(\alpha_1,\dots,\alpha_r)$ and $\beta=(\beta_1,\beta_2,\dots,\beta_q)$, we have 
$\alpha \odot \beta= (\alpha_1,\dots,\alpha_{r-1},\alpha_r+\beta_1,\beta_2,\dots,\beta_q)$ if at most one of $\alpha_r$ and $\beta_1$ is dotted (that is, two dotted entries cannot overlap). Note that $\dot a + b$ and $a +\dot b$ are both defined to be equal to $\dot c$, where $c=a+b$.  

Two new operations will prove essential to obtain the antipode of a fundamental quasisymmetric function in superspace.
Let  $\bullet$ and $\odot$ be both bilinear products defined on the monomials symmetric functions in superspace as
\begin{equation} \label{mon1}
  M_{\alpha}\bullet 
  M_{\beta}= M_{\alpha \cdot \beta}
\end{equation}  
and
\begin{equation} \label{mon2}
M_{\alpha}\odot M_{\beta}= M_{\alpha \odot \beta},
\end{equation}
which are both obviously associative and noncommutative. We note that if the last entry of $\alpha$ and the first entry of
$\beta$ are both dotted, then we simply let  $M_{\alpha}\odot M_{\beta}=0$.
\begin{remark}
 The ring of noncommutative symmetric functions in superspace, ${\rm sNSym}$, has a basis $H_\alpha$, indexed by dotted compositions, whose product is such that $H_{\alpha} H_\beta=H_{\alpha \cdot \beta}$ \cite{FLP}.  The ring ${\rm sNSym}$, which is generated by
 $$\{H_1,H_2,\dots; H_{\dot 0},H_{\dot 1},\dots \},$$
is also a Hopf algebra whose  coproduct and antipode are respectively such that \cite{FLP} 
 $$
\Delta' (H_n) = \sum_{k+\ell=n} H_k \otimes H_\ell, \qquad 
\Delta' (H_{\dot n}) = \sum_{k+\ell=n} ( H_{\dot k} \otimes H_\ell + H_{k} \otimes H_{\dot \ell})
 $$
and \cite{AGM} 
$$S'(H_{n})=  \sum_{\gamma\,  \trianglelefteq \,  (n)}(-1)^{\ell(\gamma)} H_{\gamma},
\qquad 
S'(H_{\dot n})=  \sum_{\gamma\,  \trianglelefteq \,  (\dot n)}(-1)^{\ell(\gamma)} H_{\gamma} .
$$
The ring  ${\rm sNSym}$ with this Hopf algebra structure is in fact dual to the 
 Hopf algebra structure on ${\rm sQSym}$ (whose product, coproduct and antipode are those considered in this article). In view of \eqref{mon1}, there is a dual Hopf algebra structure on the ring of quasisymmetric functions in superspace in which the product is given by $\bullet$ and where the coproduct $\Delta''$ and antipode $S''$ are simply obtained by replacing $H$ by $M$ in the previous equations:
 $$
\Delta'' (M_n) = \sum_{k+\ell=n} M_k \otimes M_\ell, \qquad 
\Delta'' (M_{\dot n}) = \sum_{k+\ell=n} (M_{\dot k} \otimes M_\ell + M_{k} \otimes M_{\dot \ell})
 $$
and 
$$S''(M_{n})=  \sum_{\gamma\,  \trianglelefteq \,  (n)}(-1)^{\ell(\gamma)} M_{\gamma},
\qquad 
S''(M_{\dot n})=  \sum_{\gamma\,  \trianglelefteq \,  (\dot n)}(-1)^{\ell(\gamma)} M_{\gamma} .
$$
We do not know whether the product $\odot$ is also associated to a natural Hopf algebra on the ring of quasisymmetric functions in superspace. 
\end{remark}

The multiplication of  
fundamental quasisymmetric functions in superspace using   these two operations
satisfies simple relations. 
\begin{proposition} \label{prop51}
  Let $\alpha$ and $\beta$ be dotted compositions. We have
     \begin{equation} \label{propo3}
    L_{\alpha}\bullet L_{\beta}= L_{\alpha \cdot \beta}.
  \end{equation}
     Moreover, if the last entry of $\alpha$ and the first entry of $\beta$ are both non-dotted, then
\begin{equation} \label{propo4}
  L_{\alpha}\bullet L_{\beta}+L_{\alpha}\odot L_{\beta}= L_{\alpha \odot \beta}.
\end{equation}  
\end{proposition}
\begin{proof} We first prove \eqref{propo3}. Using Definition~\ref{def:fundamental} and \eqref{mon1}, we obtain immediately  \eqref{propo3} since
 \begin{equation}
L_{\alpha}\bullet L_{\beta}  =
 \sum_{\mu \preceq \alpha }M_{\mu} \bullet \sum_{\eta \preceq \beta } M_{\eta}  
= \sum_{\mu \preceq \alpha } \sum_{\eta \preceq \beta } M_{\mu \cdot \eta} 
= \sum_{\kappa \preceq  \alpha \cdot \beta}  M_{\kappa} 
= L_{\alpha \cdot \beta}.
 \end{equation}

We now prove \eqref{propo4} which is a little less straightforward. 
From the definition of the fundamental quasisymmetric functions in superspace
together with $\eqref{mon1}$ and $\eqref{mon2}$, we have
\begin{equation}
  L_{\alpha}\bullet L_{\beta}+L_{\alpha}\odot L_{\beta}
  = \sum_{\gamma  \preccurlyeq \alpha   } \sum_{\eta \preccurlyeq \beta } (M_{\gamma} \bullet M_{\eta}+M_{\gamma} \odot M_{\eta} ).
\end{equation}
By assumption, the last entry of $\alpha$ and the first entry of $\beta$ are both non-dotted.  Since $\gamma  \preccurlyeq \alpha$ and $ \eta \preccurlyeq \beta$, this implies that  the last entry of $\gamma$ and the first entry of $\eta$ are also both non-dotted.  We thus deduce that
$M_\gamma \odot M_\eta$ is always equal to $M_{\gamma \odot \eta}$.  Hence \eqref{propo4} holds if we can show that
\begin{equation}
  \sum_{\sigma \preccurlyeq \alpha \odot \beta}M_{\sigma}=\sum_{\gamma  \preccurlyeq \alpha  } \sum_{\eta  \preccurlyeq \beta } (M_{\gamma \cdot \eta} +M_{\gamma \odot \eta}  ).
\end{equation}
Suppose that the dotted composition $\alpha\vdash(n_\alpha,m_\alpha)$
and let $d=n_\alpha+m_\alpha$. 
We show the stronger statement that the first (resp. second) sums on both sides of the equation
\begin{equation} \label{aprouvernew}
  \sum_{\sigma \preccurlyeq \alpha \odot \beta ; d \in D(\sigma)}M_{\sigma} +  \sum_{\sigma \preccurlyeq \alpha \odot \beta;  d \not \in D(\sigma)}M_{\sigma}= \sum_{\gamma  \preccurlyeq \alpha ,  \eta  \preccurlyeq \beta } M_{\gamma \cdot \eta}
  + \sum_{\gamma  \preccurlyeq \alpha ,  \eta  \preccurlyeq \beta } M_{\gamma \odot \eta}
  \end{equation}
are equal. 
Consider the map $\sigma \mapsto (\gamma,\eta)$, where $\gamma$ and $\eta$
are the unique dotted compositions of $(n_\alpha,m_\alpha)$ and $(n_\beta,m_\beta)$ respectively such that
$$
D(\gamma)= D(\sigma) \setminus \{ d,\dots,n+m-1 \},\quad
E(\gamma)= E(\sigma) \setminus \{ d,\dots,n+m-1 \}, \quad
F(\gamma)= F(\sigma) \setminus \{ d,\dots,n+m-1 \}
$$
and
$$
D_d(\eta)= D(\sigma) \setminus \{ 1,\dots,d \},\quad
E_d(\eta)= E(\sigma) \setminus \{ 1,\dots,d \},  \quad
F_d(\eta)= F(\sigma) \setminus \{ 1,\dots,d \}
$$
(note that on the previous line, for a given set $A=\{a_1,\dots,a_r \}$, we let
$A_d=\{a_1+d,\dots,a_r+d \}$). It is then immediate by construction that
\begin{equation} \label{eqDgamma}
D (\gamma)  \cup D_d(\eta)=D(\sigma) \setminus \{ d \}, \quad E (\gamma)  \cup E_d(\eta)=E(\sigma),
 \quad {\rm and} \quad   F(\gamma) \cup F_{d}(\eta)=F(\sigma).
\end{equation}
It is also easy to obtain that
\begin{equation} \label{eqDdot}
D (\alpha \cdot \beta) = D(\alpha) \cup D_{d}(\beta) \cup \{d\}, \quad E (\alpha \cdot \beta) = E(\alpha) \cup E_{d}(\beta), \quad {\rm and} \quad  F (\alpha \cdot \beta) = F(\alpha) \cup F_{d}(\beta).
\end{equation}
If the last entry of $\alpha$ and the first entry of $\beta$ are both non-dotted, we can similarly obtain that
\begin{equation} \label{eqDcdot}
D (\alpha \odot \beta) = D(\alpha) \cup D_{d}(\beta), \quad E (\alpha \odot \beta) = E(\alpha) \cup E_{d}(\beta), \quad {\rm and} \quad  F (\alpha \odot \beta) = F(\alpha) \cup F_{d}(\beta).
\end{equation}
We thus have from \eqref{precequiv} that
\begin{equation} \label{eqDD}
  \sigma  \preccurlyeq \alpha \odot \beta \iff D(\alpha) \cup D_d(\beta) \subseteq D(\sigma), \quad  E(\alpha) \cup E_{d}(\beta)=E(\sigma), \quad {\rm and} \quad   F(\alpha) \cup F_{d}(\beta)=F(\sigma)
\end{equation}
from which we deduce, using \eqref{eqDgamma}, that
$\gamma \preccurlyeq \alpha$ and $\eta \preccurlyeq \beta$ iff  $\sigma \preccurlyeq \alpha \odot \beta$ just as in the non-dotted case.

Finally, let $B=\{ \sigma \, | \, \alpha \odot \beta \preccurlyeq \sigma  , d \in D(\sigma) \}$ and $C=\{ \sigma \, | \, \alpha \odot \beta \preccurlyeq \sigma  , d \not \in D(\sigma) \}$. From our previous observation, we get that the map  $\sigma \mapsto (\gamma , \eta)$ restricted to $B$, which is such that $\gamma \cdot \eta=\sigma$ from \eqref{eqDgamma} and \eqref{eqDdot}, provides a bijection between the sets $B$ and
 $\{ (\gamma, \eta) \, | \, \gamma  \preccurlyeq \alpha, \eta  \preccurlyeq \beta  \}$. Similarly,   the map  $\sigma \mapsto (\gamma , \eta)$ restricted to $C$, which is such that $\gamma \odot \eta=\sigma$ from \eqref{eqDgamma} and \eqref{eqDcdot}, provides a bijection between  the sets $C$ and 
 $\{ (\gamma, \eta) \, | \, \gamma  \preccurlyeq \alpha, \eta  \preccurlyeq \beta  \}$. This proves \eqref{aprouvernew}, and completes the proof.
\end{proof}

We now show that the antipode behaves well when acting on both of our products.

\begin{theorem} \label{propantiM}  Let $F$ and $G$ be quasisymmetric functions in superspace that  are homogeneous in the fermionic variables, and let
$m_F$ and $m_G$ be their respective fermionic degrees. The antipode then satisfies the following two properties:
\begin{equation} \label{pprop2}
S(F \bullet G)= (-1)^{m_F  m_G}\bigl(S(G) \bullet S(F)+S(G) \odot S(F) \bigr)
\end{equation}
and
\begin{equation} \label{pprop1}
  S(F \odot G)= (-1)^{m_F m_G  -1}S(G) \odot S(F).
\end{equation}
\end{theorem}
\begin{proof}
 By linearity, it suffices to show the two relations on any basis of the ring of quasisymmetric functions in superspace which is homogeneous in the fermionic variables. We use the monomial quasisymmetric functions in superspace basis.  We thus need to show that for $\alpha,\beta$ dotted compositions
 of fermionic degree $m_\alpha$ and $m_\beta$ respectively, we have
 \begin{equation} \label{prop2}
S(M_{\alpha} \bullet M_{\beta})= (-1)^{m_{\alpha } m_{\beta}}\bigl(S(M_{\beta}) \bullet S(M_{\alpha})+S(M_{\beta}) \odot S(M_{\alpha}) \bigr)
\end{equation}
and
\begin{equation} \label{prop1}
  S(M_{\alpha}  \odot M_{\beta})= (-1)^{m_{\alpha }m_{\beta } -1}S(M_{\beta}) \odot S(M_{\alpha}).
\end{equation}

We first prove \eqref{prop2}.  We proceed by induction on the length of $\alpha$. Suppose that $\alpha=(n)$ is of length 1 and non-dotted for simplicity (the dotted case is similar).
From Proposition~\ref{prop:antiM}, we have
\begin{equation*}
  S(M_{(n) \cdot \beta}) =(-1)^{l+1+{m_{\beta}\choose 2}} \sum_{\sigma \unrhd (\beta_l,\dots,\beta_1,n)} M_{\sigma}
\end{equation*}  
and
\begin{align} \nonumber
  S(M_{\beta}) \bullet S(M_{(n)})+ S(M_{\beta}) \odot S(M_{(n)})&=  (-1)^{l+{m_{ \beta}\choose 2}} \left( \sum_{\gamma \unrhd ( \beta_l,\dots,\beta_1)} \left[ M_{\gamma} \bullet  \left(-  M_{(n)} \right) +  M_{\gamma}  \odot \left(-  M_{(n)} \right) \right]  \right)\\
&= (-1)^{l+1+{m_{ \beta}\choose 2}} \sum_{\gamma \unrhd ( \beta_l,\dots,\beta_1)}  (M_{\gamma \cdot \eta}+M_{\gamma \odot \eta}) ,\nonumber
\end{align}
where we made use of \eqref{mon1} and \eqref{mon2}.
Proving \eqref{prop2} in the case $\alpha=(n)$ thus amounts to showing that
\begin{equation} \label{eq513}
  \sum_{\sigma \unrhd (\beta_l,\dots, \beta_1,n)} M_{\sigma} =  \sum_{\gamma \unrhd (\beta_l,\dots,\beta_1)}  (M_{\gamma \cdot (n)}+M_{\gamma \odot (n)}).
\end{equation}
Observe that the last entry of $\sigma$, denoted $\sigma_r$, is such that $\sigma_r \geq n$ since
$\sigma \unrhd (\beta_l,\dots, \beta_1,n)$. We can thus rewrite the previous equation as
\begin{equation} \label{eq513b}
  \sum_{\sigma \unrhd (\beta_l,\dots, \beta_1,n); \sigma_r =n} M_{\sigma} +   \sum_{\sigma \unrhd (\beta_l,\dots, \beta_1,n); \sigma_r \neq n} M_{\sigma} =  \sum_{\gamma \unrhd (\beta_l,\dots,\beta_1)}  M_{\gamma \cdot (n)}+  \sum_{\gamma \unrhd (\beta_l,\dots,\beta_1)}  M_{\gamma \odot (n)}.
\end{equation}
Using $\sigma =\gamma \cdot (n)$, where $\gamma=(\sigma_1,\dots,\sigma_{r-1})$,
it is then obvious that the first sums on both sides are equal, while using
$\sigma =\gamma \odot (n)$, where $\gamma=(\sigma_1,\dots,\sigma_{r}-n)$, we get that the second sums on both sides are equal.

We can now prove  \eqref{prop2} in the  case where $\alpha$ is of arbitrary length.  Suppose that $\alpha_1$ is non-dotted.  Letting   $\hat \alpha =(\alpha_2,\dots,\alpha_r)$, we have from induction   that
\begin{align*}
  S(M_\alpha \bullet M_\beta)=S(M_{(\alpha_1)} \bullet M_{\hat \alpha} \bullet M_\beta) 
  & = S( M_{\hat \alpha} \bullet M_\beta) \bullet S(M_{(\alpha_1)}) +  S( M_{\hat \alpha} \bullet M_\beta) \odot S(M_{(\alpha_1)})   \\
  &= (-1)^{m_\alpha m_\beta} \bigl( S(M_\beta) \bullet S(M_{\hat \alpha}) +  S(M_\beta) \odot S(M_{\hat \alpha}) \bigr) \bullet S(M_{(\alpha_1)}) \\
  &  \qquad + (-1)^{m_\alpha m_\beta} \bigl( S(M_\beta) \bullet S(M_{\hat \alpha}) +  S(M_\beta) \odot S(M_{\hat \alpha}) \bigr) \odot S(M_{(\alpha_1)}) \\
&= (-1)^{m_\alpha m_\beta} S(M_\beta) \bullet \bigl( S(M_{\hat \alpha}) \bullet S(M_{(\alpha_1)}) +  S(M_{\hat \alpha}) \odot S(M_{(\alpha_1)}) \bigr)  \\
  &  \qquad +(-1)^{m_\alpha m_\beta} S(M_\beta) \odot \bigl( S(M_{\hat \alpha}) \bullet S(M_{(\alpha_1)}) +  S(M_{\hat \alpha}) \odot S(M_{(\alpha_1)}) \bigr)
 \\
  &=  (-1)^{m_\alpha m_\beta} \bigl(  S(M_\beta) \bullet  
  S( M_{\alpha_1} \bullet M_{\hat \alpha})  +  S(M_\beta) \odot  
  S( M_{\alpha_1} \bullet M_{\hat \alpha})  \bigr) 
  \\
  &=  (-1)^{m_\alpha m_\beta} \bigl(  S(M_\beta) \bullet 
  S( M_{\alpha})  +  S(M_\beta) \odot  
  S( M_{\alpha})  \bigr) ,
\end{align*}
as desired.
In the case where $\alpha_1$ is dotted, the equations are the same except that 
the signs in the last five equalites are respectively $(-1)^{m_\alpha+ m_\beta-1 }$, $(-1)^{m_\alpha +m_\beta -1 +(m_\alpha -1)m_\beta }=(-1)^{m_\alpha m_\beta +m_\alpha -1}$, $(-1)^{m_\alpha m_\beta +m_\alpha -1}$, $(-1)^{m_\alpha m_\beta +m_\alpha -1+(m_\alpha-1)}=(-1)^{m_\alpha m_\beta}$
and $(-1)^{m_\alpha m_\beta }$, which means that the result also holds.

We now prove \eqref{prop1}.  We begin the proof by considering the case where
$M_{\alpha}  \odot M_{\beta}=0$.  By definition of the $\odot$ product, 
we have in this case that the last entry of $\alpha$  and the first entry of $\beta$ are both dotted. Using Proposition~\ref{prop:antiM} and the definition of the order $ \unrhd $, it is easy to see that all the compositions $\gamma$ appearing in $S(M_\beta)$ have their {\it last} entry dotted and that all the compositions $\omega$ appearing in $S(M_\alpha)$ have their {\it first} entry dotted. Hence $S(M_{\beta}) \odot S(M_{\alpha})=0$ and the result holds in that case since
$$
 S(M_{\alpha}  \odot M_{\beta})=S(0)=0 = (-1)^{m_{\alpha }m_{\beta } -1}S(M_{\beta}) \odot S(M_{\alpha}).
$$
 We now consider the case when $M_{\alpha}  \odot M_{\beta}=M_{\alpha \odot \beta}\neq 0$, that is when
 the last entry of $\alpha$  and the first entry of $\beta$ are not  both dotted.  We will again prove the result by induction on the length of $(\alpha)$.  We will consider only the case $\alpha=(n)$ as the case $\alpha=(\dot n)$ is similar.
From Proposition~\ref{prop:antiM}, we have 
\begin{equation} \label{prop1a1}
  S(M_{(n) \odot \beta})=S(M_{(n+\beta_1,\dots,\beta_l)}) =(-1)^{l+{m_{ \beta}\choose 2}} \sum_{\sigma \unrhd (\beta_l,\dots,\beta_1+n)} M_{\sigma}.
\end{equation}  
On the other hand, using again Proposition~\ref{prop:antiM}, we obtain from    
\eqref{mon2} that
\begin{align} \nonumber
 S(M_{\beta}) \odot S(M_{(n)}) &=   \left( (-1)^{l+{m_{ \beta}\choose 2}} \sum_{\omega \unrhd (\beta_l,\dots,\beta_1)} M_{\omega}\right) \odot \left( -M_{(n)} \right)\\ \label{prop1a2} 
  &= (-1)^{ l+{m_{\beta }\choose 2}-1} \sum_{\omega \unrhd ( \beta_l,\dots,\beta_1)}  M_{\omega \odot (n)}.
\end{align}
Comparing 
\eqref{prop1a1} and \eqref{prop1a2}, we thus get that \eqref{prop1} will hold
in the case $\alpha=(n)$
if we can show that
\begin{equation}
  \sum_{\sigma \unrhd (\beta_l,\dots,\beta_1+n)} M_{\sigma} =  \sum_{\omega \unrhd (\beta_l\dots,\beta_1)}  M_{\omega \odot (n)}.
\end{equation}
But the equality is easily seen to hold using $\sigma=\omega \odot (n)$, where $\omega=(\sigma_1,\dots,\sigma_r-n)$.

Suppose finally that the length of $\alpha$ is arbitrary in \eqref{prop1} and that $\alpha_1$ is non-dotted. 
Letting again   $\hat \alpha =(\alpha_2,\dots,\alpha_r)$, we have from induction and  \eqref{prop2} that
\begin{align*}
  S(M_\alpha \odot M_\beta)=S(M_{(\alpha_1)} \bullet M_{\hat \alpha} \odot M_\beta) 
  & = S( M_{\hat \alpha} \odot M_\beta) \bullet S(M_{(\alpha_1)}) +  S( M_{\hat \alpha} \odot M_\beta) \odot S(M_{(\alpha_1)})   \\
  &= (-1)^{m_\alpha m_\beta-1} S(M_\beta) \odot \Bigl( 
  S( M_{\hat \alpha}) \bullet S(M_{(\alpha_1)}) +   S( M_{\hat \alpha})  \odot S(M_{(\alpha_1)})  \Bigr ) \\
  &=  (-1)^{m_\alpha m_\beta-1} S(M_\beta) \odot  
  S( M_{\alpha}) ,
\end{align*}
as desired. In the case where $\alpha_1$ is dotted, the signs in the last three equalites are respectively $(-1)^{m_\alpha +m_\beta -1}$, $(-1)^{m_\alpha +m_\beta -1 +(m_\alpha -1)m_\beta -1}=(-1)^{m_\alpha m_\beta +m_\alpha -1}$ and $(-1)^{m_\alpha m_\beta -1}$, which implies that the result is also seen to hold.

\end{proof}  
We obtain as a corollary the main property that will allow us to compute the antipode of any fundamental quasisymmetric function in superspace recursively.
\begin{corollary} \label{coroS}
  For $\alpha$ and $\beta$ dotted compositions such that
  the last entry of $\alpha$ and the first entry of $\beta$ are both non-dotted, we have
\begin{equation} \label{eqdecompL}
  S(L_{\alpha \odot \beta})=(-1)^{m_{\alpha } m_{\beta}} S(L_{\beta}) \bullet S(L_{\alpha}),
\end{equation}
where $m_\alpha$ and $m_\beta$ are the fermionic degrees of $\alpha$ and $\beta$ respectively.  
\end{corollary} 
\begin{proof}
  From Proposition~\ref{prop51}, we have when  the last entry of $\alpha$ and the first entry of $\beta$ are both non-dotted that
  $$
L_{\alpha \odot \beta} = L_\alpha \bullet L_\beta +  L_\alpha \odot L_\beta .
  $$
Using Theorem~\ref{propantiM}, we then conclude that
$$
(-1)^{m_{\alpha } m_{\beta}}  S(L_{\alpha \odot \beta} )  = \bigl(S(L_{\beta}) \bullet S(L_{\alpha}) + S(L_{\beta}) \odot S(L_{\alpha} ) \bigr) - S(L_{\beta}) \odot S(L_{\alpha}) 
  = S(L_{\beta}) \bullet S(L_{\alpha}),
  $$
which is equivalent to \eqref{eqdecompL}.  
\end{proof}
It is now simple to compute the antipode of a fundamental quasisymmetric function in superspace recursively.  We say that a dotted composition is a
 {\defstyle column}
  if all of its non-dotted entries are equal to 1.  It is easy to see
that any dotted composition $\gamma$ can be decomposed uniquely as
\begin{equation} \label{columndecomp}
\gamma = \alpha^{(1)} \odot \alpha^{(2)} \cdots \odot \alpha^{(\ell)} ,
\end{equation}
where all the dotted compositions $\alpha^{(1)},\dots, \alpha^{(\ell)}$ are columns whose first and last entries are equal to 1 (except possibly the first entry of $\alpha^{(1)}$ and the last entry of $\alpha^{(\ell)}$). For instance, we have
$$
(2,\dot 0, 3, 1,\dot 3,1,\dot 2,2,1,3,\dot 0)=(1) \odot (1,\dot 0,1) \odot (1) \odot (1,1,\dot 3 ,1,\dot 2,1) \odot (1,1)\odot (1) \odot (1,\dot 0)   .
$$
Using the decomposition  \eqref{columndecomp}, we
get from Corollary~\ref{coroS} that the antipode of a fundamental quasisymmetric function in superspace decomposes as
\begin{equation} \label{eqdecomp}
S(L_\gamma)= (-1)^{ {\rm sign(\gamma)} } S (L_{\alpha^{(\ell)}}) \bullet S (L_{\alpha^{(\ell-1)}}) \bullet \cdots \bullet S (L_{\alpha^{(1)}}),
\end{equation}
where ${\rm sign} (\gamma) =\sum_{i<j} m_{\alpha^{(i)}} m_{\alpha^{(j)}}$.
It thus only remains to obtain $S(L_\alpha)$ when the dotted composition $\alpha$ is a column. For that purpose, let a dotted composition be  {\defstyle maximal}
 if it does not have any consecutive non-dotted entries.  For example, the composition 
$
(1,\dot 2, 4, \dot 5,\dot 4, 7, \dot 0, 2, \dot 3)
$
is maximal  while 
$
(1,\dot 2, 4, 2,\dot 5,\dot 4, 7, \dot 0, 2, \dot 3)$ and $(2,1,\dot 2,4,\dot 5,\dot 4, 7, \dot 0, 2, \dot 3)$ are not.  It is not hard to see that
any composition $\alpha$ is such that $\alpha  \preccurlyeq \beta$ for a unique maximal dotted composition $\beta$ (obtained by summing all the consecutive non-dotted entries of $\alpha$).

The following proposition gives the desired action of $S$ on $L_\alpha$ when $\alpha$ is a column. 
\begin{proposition} \label{propcolumn}  Let  the dotted composition $\alpha$ be a column. We then have
  $$S(L_{\alpha})=(-1)^{\ell(\alpha)+{m_{\alpha }\choose 2}} \sum_{\beta} L_{\beta}, $$
  where the sum is over all maximal dotted compositions $\beta$ such that $\beta \unrhd \Rev(\alpha)$.
\end{proposition}

\begin{proof} We first observe that if $\alpha$ is a column then
  \begin{equation}
    \beta \unrhd \Rev(\alpha) \quad {\rm and} \quad \delta  \preccurlyeq \beta \quad \iff \quad \delta \unrhd \Rev(\alpha) ,
  \end{equation}
  where $\beta$ is the unique maximal dotted composition such that $\delta  \preccurlyeq \beta$.  The reverse implication is trivial while the direct one holds because
$\delta$ is obtained from $\beta$ by splitting its non-dotted entries and that all the non-dotted entries of $\alpha$ are 1's by definition.

  Now, from the definition of $L_\alpha$, it is immediate that
$L_\alpha=M_\alpha$ if $\alpha$ is a column.   Hence, we have from Proposition~\ref{prop:antiM} and our previous observation that
  $$S(L_{\alpha}) =S ( M_{\alpha})=(-1)^{\ell(\alpha)+{m_{\alpha }\choose 2}} \sum_{\delta \unrhd \Rev(\alpha)} M_{\delta}= (-1)^{\ell(\alpha)+{m_{\alpha }\choose 2}} \sum_{\beta} \sum_{\delta  \preccurlyeq \beta }M_{\delta} = (-1)^{\ell(\alpha)+{m_{\alpha }\choose 2}} \sum_{\beta} L_\beta, $$
  where, in the double sum, the first  sum is over all maximal dotted compositions $\beta$ such that $\beta \unrhd \Rev(\alpha)$.
     \end{proof}    
For instance, considering the column $\alpha=(1^2, \dot{5}, 1^3)$, we have 
\begin{align*} 
S\left(L_{(1^2, \dot{5}, 1^3)}\right)&=(-1)^{6+{1 \choose 2}}\bigl(L_{(3, \dot{5}, 2)}+ L_{(3, \dot{6}, 1)}+ L_{(3, \dot{7} )} + L_{(2, \dot{6}, 2)} + L_{(2, \dot{7}, 1)}+ L_{(2, \dot{8})}+ L_{(1, \dot{7}, 2)}+ L_{(1, \dot{8}, 1)} \\
&+L_{(1, \dot{9})} + L_{(\dot{8}, 2)}  +  L_{(\dot{9}, 1)} + L_{( \dot{10})} \bigr) ,
\end{align*} 
where we see that the indices are
all the maximal dotted compositions $\beta$ such that $\beta \unrhd \Rev(\alpha)=(1^3,\dot 5, 1^2)$.

\begin{remark} In the non-dotted case, columns are simply compositions of the form $(1^\ell)$ for $\ell=1,2,3,\dots$.  Since $S(L_{(1^\ell)})=(-1)^\ell L_{(\ell)}$ has only one term, 
we recover from our decomposition the usual formula for the antipode of a fundamental quasisymmetric function.  We will illustrate this with an example.  Let $\alpha=(3,1,2,1,2)$.  The column decomposition \eqref{columndecomp}
  is then
  $$
(3,1,2,1,2) = (1) \odot (1) \odot (1,1,1) \odot (1,1,1) \odot (1).
  $$
  This gives
  $$
S(L_{(3,1,2,1,2)}) = S(L_{(1)})  \bullet S(L_{(1,1,1)}) \bullet  S(L_{(1,1,1)}) \bullet
 S(L_{(1)}) \bullet  S(L_{(1)}) . $$
 Using $S(L_{(1)})= -L_{(1)}$ and  $S(L_{(1,1,1)})= -L_{(3)}$, we then get from \eqref{propo3} that
  $$
S(L_{(3,1,2,1,2)}) = (-1)^5 L_{(1)}  \bullet L_{(3)} \bullet  L_{(3)} \bullet
 L_{(1)} \bullet  L_{(1)} = - L_{(1,3,3,1,1)}, $$
 which coincides with the known formula for the action of the antipode on
 $L_{(3,1,2,1,2)}$.
\end{remark}  

\section{Schur function expansion}
\label{sec:schurFnsInTermsOfFunds}

\subsection{Schur functions $ s_{\Lambda/\Omega}$ in terms of fundamentals}
The expansion of Schur functions  in terms of fundamental quasisymmetric functions can be given as a sum over standard tableaux.  We show that this result extends to superspace.

We  first recall some definitions
related to partitions \cite{Mac}.
A partition $\lambda=(\lambda_1,\lambda_2,\dots)$ of degree $|\lambda|$
is a vector of non-negative integers such that
$\lambda_i \geq \lambda_{i+1}$ for $i=1,2,\dots$ and such that
$\sum_i \lambda_i=|\lambda|$.  
Each partition $\lambda$ has an associated Ferrers diagram
with $\lambda_i$ lattice squares in the $i^{th}$ row,
from the top to bottom. Any lattice square in the Ferrers diagram
is called a cell (or simply a square), where the cell $(i,j)$ is in the $i$th row and $j$th
column of the diagram.  
We say that the diagram $\mu$ is contained in $\lambda$, denoted
$\mu\subseteq \lambda$, if $\mu_i\leq \lambda_i$ for all $i$.  Finally,
$\lambda/\mu$ is a horizontal (resp. vertical) $n$-strip if $\mu \subseteq \lambda$, $|\lambda|-|\mu|=n$,
and the skew diagram $\lambda/\mu$ does not have two cells in the same column
(resp. row).

 A superpartition $\Lambda$ \cite{JL} is 
 a pair of partitions $(\Lambda^a; \Lambda^s)$,
  where $\Lambda^a$ is a partition with $m$ 
distinct parts (one of them possibly  equal to zero),
and $\Lambda^s$ is an ordinary partition.  A diagrammatic representation of $\Lambda$ is obtained by reordering the entries of the concatenation of $\Lambda^a$ and $\Lambda^s$ and then adding a circle at the end of every row corresponding to an entry of  $\Lambda^a$. For instance, the diagram corresponding to $\Lambda=(3,0;5,3,2)$ is
$$
{\small {\tableau[scY]{&&&& \\&  &  & \bl\tcercle{} \\&  & \\ & \\ \bl\tcercle{}}}} .
$$

The {\defstyle Schur functions in superspace} $s_{\Omega/\Lambda}$ are presented as a generating sum of certain tableaux, called $s$-tableaux, that we now define \cite{JL}.
We say that $\Omega/\Lambda$ is a {\defstyle bosonic horizontal $\ell$-strip}
of type $s$
if
\begin{enumerate}
\item $\Omega^*/\Lambda^*$ is a horizontal $\ell$-strip. 
\item  The $i$-th circle, starting from below, of $\Omega$ is either 
\begin{itemize}
\item in the same
row  as the $i$-th circle of $\Lambda$ if $\Omega^*/\Lambda^*$ does not contain a cell in that row, or
\item 
one row below that of the $i$-th circle of $\Lambda$ if $\Omega^*/\Lambda^*$ contains a cell in the row  of the
$i$-th circle of $\Lambda$ (we say in this case that the circle was moved). 
\end{itemize}
\end{enumerate}
Similarly, we say that $\Omega/\Lambda$ is a {\defstyle fermionic horizontal
$\ell$-strip} of type $s$
if
\begin{enumerate}
\item 
$\Omega^*/\Lambda^*$ is a horizontal $\ell$-strip. 
\item There exists a unique circle of $\Omega$ (the new circle), let's say in column $c$, such that
\begin{itemize}
\item column $c$  does not contain any cell
of $\Omega^*/\Lambda^*$, and 
\item there is a cell of $\Omega^*/\Lambda^*$ in 
every column strictly to the left of column $c$. 
\end{itemize}
\item If $\tilde \Omega$ is $\Omega$ without its new circle,
then the $i$-th circle, starting from below, of $\tilde \Omega$ is either 
\begin{itemize}
\item in the same
row as the $i$-th circle of $\Lambda$ if $\Omega^*/\Lambda^*$ does not contain a cell in that row,  or 
\item one row below that of the $i$-th circle of $\Lambda$ if $\Omega^*/\Lambda^*$ contains a cell in the row of the
$i$-th circle of $\Lambda$ (we say in this case that the circle was moved). 
\end{itemize}
\end{enumerate}
The sequence $\Omega=\Lambda_{(0)},\Lambda_{(1)},\dots, \Lambda_{(n)}=\Lambda$ is an $s$-tableau 
of shape $\Lambda/\Omega$ and weight $(\alpha_1,\dots,\alpha_n)$, where 
$\alpha_i\in \mathbb \{0,\dot 0, 1,\dot 1,2,\dot 2,\dots \}$, if $\Lambda_{(i)}/\Lambda_{(i-1)}$  is a horizontal $\alpha_i$-strip of type $s$ (fermionic or bosonic depending on whether $\alpha_i$ is dotted or not).
An $s$-tableau can be represented by a diagram constructed recursively in the following way:
\begin{enumerate}
\item  the cells of $\Lambda_{(i)}^*/\Lambda_{(i-1)}^*$, which form a horizontal strip,  are filled with the letter $i$.  In 
the fermionic case, the new circle is also filled with a letter $i$. 
\item the circles of $\Lambda_{(i-1)}$ that are moved a row below keep their fillings.
\end{enumerate}
For instance, the $s$-tableau 
$$
{\small {\tableau[scY]{1&1&2&3&5\\3&3&4\\5&\bl\tcercle{5}\\ \bl\tcercle{3}}}}
$$
corresponds to the sequence
$$
{\small {\tableau[scY]{&}}}
\qquad
{\small {\tableau[scY]{&&}}}
\qquad
{\small {\tableau[scY]{&&& \\&  &\bl\tcercle{}}}}
\qquad
{\small {\tableau[scY]{&&& \\&  & \\ \bl\tcercle{}}}}
\qquad
{\small {\tableau[scY]{&&&& \\&  & \\ & \bl\tcercle{}\\ \bl\tcercle{}}}}.
$$

The sign of an $s$-tableau $T$ is defined in the following way.
Read the fillings of the circles from top to bottom
to obtain a word (without repetition).
The sign of the tableau $T$ is  then equal to $(-1)^{{\rm inv} (T)}$, 
where ${{\rm inv} (T)}$ is the number of inversions of the word.
In the previous example, the sign of the tableau is $(-1)^1$ since 
the word $53$ has one inversion.

 Now, define the skew Schur function in superspace $s_{\Lambda/\Omega}$ as
\begin{equation}
  s_{\Lambda/\Omega} = \sum_{T} (-1)^{{\rm inv}(T)} 
   \te_{1}^{\eta_1} \te_{2}^{\eta_2} \cdots \te_{l}^{\eta_l} x_{1}^{\ga_1}\cdots
 x_{l}^{\ga_l},
\end{equation}
where the sum is over all $s$-tableaux of shape $\Lambda/\Omega$, and where
$T$ is of weight $(\alpha_1,\dots,\alpha_l)$.

\begin{definition}
We say that an $s$-tableau $T$ is {\defstyle dot-standard} if its weight
$(\alpha_1,\dots,\alpha_n)$ is such that $\alpha_i=1$ whenever $\alpha_i$ is non-dotted.
The non-dotted letter $i$ in $T$ is a descent of $T$ if either $\alpha_{i+1}$ is dotted or the letter $i+1$ is below $i$ in $T$.
Suppose $d_1<d_2<\cdots <d_k$ are the descents of $T$ and that there are $N$ non-dotted entries in $\alpha$.  Then let
$d'_i=d_i-r_i$, where $r_i$ is the number of dotted letters smaller than $d_i$
in $T$.
We define the composition $\comp(T)$ in the following
way. The non-dotted parts of $\comp(T)$ are
$(d_1',d_2'-d_1',\dots,d_{k}'-d_{k-1}',N-d_k')$.
If the dotted part
$\alpha_i$ follows the non-dotted part $\alpha_{i-1}$ (associated to
descent $j$) in $\alpha$,
then insert $\alpha_i$ after
$d_j'-d_{j-1}'$ in $\comp(T)$. If a dotted part follows another dotted
part in $\alpha$, then it should also follow it in $\comp(T)$.
\end{definition}

\begin{example}
Here is an example of a dot standard $s$-tableau $T$:
$$
T={\small {\tableau[scY]{1&2&4&6\\3&3&7&8\\5&\bl\tcercle{3}\\9 \\
    \bl\tcercle{8}}}} .
$$
The weight of $T$ is $(1,1,\dot 2,1,1,1,1,\dot 1,1)$, its descents are
$d_1=2,d_2=4,d_3=6,d_4=7$ with $r_1=0$, $r_2=1$, $r_3=1$ and $r_4=1$.  Hence,
$(d_1',d_2'-d_1',d_3'-d_2',d_4'-d_3',N-d_4')=(2,1,2,1,1)$ which leads to
${\rm comp}(T)=(2,\dot 2, 1,2,1,\dot 1,1)$.
\end{example}

The following proposition  gives the expansion of the skew Schur functions in superspace $s_{\Lambda/\Omega}$ in terms of fundamental quasisymmetric functions in superspace $L_\alpha$.  Before proving the proposition we define
the {\defstyle standardisation} ${\rm std}(T)$ of an $s$-tableau $T$.  
The dot-standard $s$-tableau ${\rm std}(T)$ is obtained by doing repeatedly the following operation on all the non-dotted letter in $T$: if $i$ is a non-dotted letter of $T$ with $r$ occurences,
change the $i$'s from left to right into $i$, $i+1$, $i+2$,..., $i+r-1$ and then
change every letter $j>i$ in $T$ to $j+r-1$.  Once this operation is done, if the resulting $s$-tableau $T'$ has consecutive letters, then define ${\rm std}(T)=T'$. Otherwise, relabel the letters of $T'$ so that they are consecutive starting from the letter 1 and let ${\rm std}(T)$ be the resulting $s$-tableau.
It is not difficult to see that ${\rm std}(T)$ is a dot-standard $s$-tableau of the same shape as $T$.  This is because if a circle is moved to the next row by the letter $i$ in $T$, then it will
also be moved to the next row when the $i$'s are replaced by $i$, $i+1$,..., $i+r-1$ (after the letter $i+\ell$ moves the circle to the next row, it will not be moved again later by an $i+s$ with $s>\ell$ since those letters are to the right of $i+\ell$).
\begin{proposition} \label{propsymfun} We have
\begin{equation}
s_{\Lambda/\Omega}= \sum_{\mathcal T}  (-1)^{{\rm inv}(\mathcal T)} L_{\comp(\mathcal T)} ,
\end{equation}
where the sum is over all dot-standard $s$-tableaux $\mathcal T$ of shape ${\Lambda/\Omega}$.
\end{proposition}
\begin{proof}
  Given an $s$-tableau $T$, let $\beta$ be the weight of $T$ and let
  $\mathcal T={\rm std}(T)$.  Then, let $\gamma$ be the vector obtained by removing the  0 entries in $\beta$.  It is immediate that
  $\gamma
  \preccurlyeq  {\comp(\mathcal T)}$ (a letter appearing many times in $T$ will not produce any descent during the standardisation process leading to
  ${\rm std}(T)$),
which means that  the monomial corresponding to $T$ appears in $ (-1)^{{\rm inv}(\mathcal T)}L_{\comp(\mathcal T)}$ (with the same sign since the standardisation does no change the order of the circles).
On the other hand, if $\beta  \preccurlyeq  {\comp(\mathcal T)}$, then a
tableau $T$ of weight $\beta$ can be obtained by merging and relabeling the letters of
$\mathcal T$ to form a tableau of weight $\beta$. That is, every monomial in
$L_{\comp(\mathcal T)}$ corresponds to a monomial associated to a given tableau $T$. 
\end{proof}

\begin{example} Let $\Omega$ be the empty superpartition and let
  $\Lambda=(1;2,2)$.  There are 10 standard $s$-tableaux of shape
  $(1;2,2)$:
$$
{\small {\tableau[scY]{1&3\\2&5\\4&\bl\tcercle{1}}}}\,  \qquad {\small {\tableau[scY]{1&4\\2&5\\3&\bl\tcercle{1}}}}\,  \qquad {\small {\tableau[scY]{1&3\\2&5\\4&\bl\tcercle{2}}}}\,   \qquad {\small {\tableau[scY]{1&4\\2&5\\3&\bl\tcercle{2}}}}\,   \qquad {\small {\tableau[scY]{1&2\\3&5\\4&\bl\tcercle{3}}}}\,   \qquad {\small {\tableau[scY]{1&4\\2&5\\3&\bl\tcercle{3}}}}\,  \qquad {\small {\tableau[scY]{1&2\\3&5\\4&\bl\tcercle{4}}}}\,   \qquad {\small {\tableau[scY]{1&3\\2&5\\4&\bl\tcercle{4}}}}\,   \qquad {\small {\tableau[scY]{1&2\\3&4\\5&\bl\tcercle{5}}}}\,  \qquad {\small {\tableau[scY]{1&3\\2&4\\5&\bl\tcercle{5}}}}
$$
which gives
$$
s_{(1;2,2)}= L_{(\dot 1 ,2 ,2)} + L_{(\dot 1 ,1,2 ,1)}+  L_{(1, \dot 1 ,1,2)}+ L_{(1, \dot 1 ,2,1)}+ L_{(2, \dot 1 ,2)} + L_{(1,1, \dot 1 ,1,1)} +L_{(2,1, \dot 1 ,1)} + L_{(1,2, \dot 1 ,1)}+L_{(2, 2,\dot 1 )} + L_{(1, 2, 1,\dot 1 )}.
$$

\end{example}

\bibliographystyle{plain}
\bibliography{superQ}
\label{sec:biblio}

\end{document}